\documentclass[11pt,a4paper]{amsart}
\usepackage{fullpage}
\usepackage{color}
\usepackage{epsfig, amsmath,xspace}
\usepackage{amssymb,amscd,mathrsfs}
\usepackage[all,cmtip]{xy}
\usepackage{hyperref}

\newtheorem{thm}{Theorem}[section]
\newtheorem{lemma}[thm]{Lemma}
\newtheorem{proposition}[thm]{Proposition}
\newtheorem{corollary}[thm]{Corollary}
\theoremstyle{definition}
\newtheorem{remark}[thm]{Remark}
\newtheorem{defn}[thm]{Definition}
\newtheorem{example}[thm]{Example}

\newdir{ >}{{}*!/-7pt/@{>}}

\def\ie{\emph{i.e.\,\,}}
\def\eg{\emph{e.g.\,\,}}
\def\leq{\leqslant}
\def\geq{\geqslant}

\def\id{\mathrm{id}}

\def\ra{\rightarrow}

\DeclareMathOperator{\colim}{colim}
\renewcommand{\Im}{{\rm Im}}

\newcommand{\Der}{{ \rm Der}}
\newcommand{\Hom}{{\rm Hom}}

\newcommand{\Bij}{{\rm Bij}}

\newcommand{\dgmod}{dg\text{-}\mathrm{mod}}

\newcommand{\mcE}{\mathcal{E}}
\newcommand{\mcF}{\mathcal{F}}
\newcommand{\mcK}{\mathcal{K}}

\newcommand{\mcO}{\mathcal{O}}
\newcommand{\mcP}{\mathcal{P}}
\newcommand{\mcQ}{\mathcal{Q}}
\newcommand{\mcR}{\mathcal{R}}
\newcommand{\mcS}{\mathcal{S}}
\newcommand{\As}{{\mathcal A}s}
\newcommand{\Com}{{\mathcal C}om}

\newcommand{\sh}{{ \rm sh}}
\newcommand{\const}{\mathrm{const}}
\newcommand{\Res}{\mathrm{res}}
\newcommand{\ev}{\mathrm{ev}}

\newcommand{\KO}{{}_{\mcK}\mcO}
\newcommand{\UCModMC}{{}_{U_{\Com}}\mathrm{Mod}(\mathcal{M}_{\Com})}
\newcommand{\UCModME}{{}_{U_{\Com}}\mathrm{Mod}(\mathcal{M}_E)}
\newcommand{\UCModMEn}{{}_{U_{\Com}}\mathrm{Mod}(\mathcal{M}_{E_n})}
\newcommand{\UModP}{{}_U\mathrm{Mod}(\mathcal{M}_{\mcP})}
\newcommand{\VModQ}{{}_V\mathrm{Mod}(\mathcal{M}_{\mcQ})}
\newcommand{\MP}{\mathcal{M}_{\mcP}}

\newcommand{\tree}{ \xymatrix{[r_n] \ar[r]^{f_n} & \cdots\ar[r]^{f_2}&[r_1]}}
\newcommand{\bB}{\bar{B}}

\begin{document}
\title{Iterated bar complexes and $E_n$-homology with coefficients}
\author{Benoit Fresse}
\author{Stephanie Ziegenhagen}
\address{UMR 8524 de l'Universit\'e Lille 1 - Sciences et Technologies - et du CNRS,
Cit\'e Scientifique B\^{a}timent M2,
F-59655 Villeneuve d'Ascq C\'edex, France}
\email{Benoit.Fresse@math.univ-lille1.fr}
\address{Universit\'e Paris 13, Sorbonne Paris Cit\'e, LAGA, CNRS (UMR 7539)\\
 99 Avenue Jean-Baptiste Cl\'ement, 93430 Villetaneuse, France}
\email{ziegenhagen@math.univ-paris13.fr}

\thanks{The first author is supported in part by grant ANR-11-BS01-002 ``HOGT'' and by the Labex ANR-11-LABX-0007-01 ``CEMPI''. The second author gratefully acknowledges support by the DFG}

\begin{abstract}
The first author proved in a previous paper that the $n$-fold bar construction for commutative algebras can be generalized to $E_n$-algebras, and that one can calculate $E_n$-homology with trivial coefficients via this iterated bar construction. We extend this result to $E_n$-homology and $E_n$-cohomology of a commutative algebra $A$ with coefficients in a symmetric $A$-bimodule.
\end{abstract}

\maketitle

\section*{Introduction}

Boardman-Vogt \cite{BV73} and May \cite{Ma72} showed in the early 70's that $n$-fold loop spaces
are equivalent to algebras over the little $n$-cubes operad (at least, up to group completion issues,
or provided that we restrict ourselves to connected spaces).
In this paper, we deal with $E_n$-operads in chain complexes, which are defined as operads quasi-isomorphic to the chain operad
of little $n$-cubes, for any $1\leq n\leq\infty$.
The category of $E_n$-algebras, algebras over an $E_n$-operad, can be thought of as an algebraic version of the category of $n$-fold loop spaces.
If we choose a $\Sigma_*$-cofibrant $E_n$-operad, let $E_n$, then we can define homological invariants specifically suited to $E_n$-algebras,
namely the $E_n$-homology and $E_n$-cohomology of an $E_n$-algebra $A$ with coefficients in a representation $M$ of $A$.
These invariants are defined via derived Quillen functors.
Note that every strictly commutative algebra $A$ is an $E_n$-algebra for any $1\leq n\leq\infty$, and that every symmetric $A$-bimodule
is a representation of the $E_n$-algebra $A$.
The homology of $E_n$-algebras can also be regarded as a particular case of the factorization homology theory (see~\cite{Fra13}),
and, in the case of commutative algebras,
agrees with Pirashvili's higher Hochschild homology theory (see~\cite{Pir00}).
In the case $n=\infty$, we retrieve the $\Gamma$-homology theory of~\cite{RW02}.

In \cite{Fre11} the first author also proved that the classical iterated bar complex of augmented commutative algebras $B^n(A) = B\circ\dots\circ B(A)$,
such as defined by Eilenberg and MacLane (see~\cite{EM53}), computes the $E_n$-homology of $A$
with coefficients in the commutative unital ground ring $M = k$.
The purpose of this paper is to extend this construction to the non-trivial coefficient case in order to get an explicit chain complex
computing the $E_n$-homology and the $E_n$-cohomology of commutative
algebras with coefficients.
In short, we define a chain complex $(M\otimes B^n(A),\partial_{\theta})$
by adding a twisting differential $\partial_{\theta}: M\otimes B^n(A)\rightarrow M\otimes B^n(A)$
to the tensor product of the Eilenberg-MacLane iterated bar complex $B^n(A)$
with any symmetric $A$-bimodule of coefficients $M$ (see Definition~\ref{defn:twistingdifferential}).  We then prove that the chain complex $B^{[n]}_*(A,M) = (M\otimes \Sigma^{-n}B^n(A),\partial_{\theta})$,
where we  use an $n$-fold desuspension $\Sigma^{-n}$ to mark an extra degree shift,
computes the $E_n$-homology of $A$ with coefficients in $M$ (see Theorem~\ref{mainresult}).
Thus, we get
\begin{equation*}
H_*^{E_n}(A,M) = H_*(M\otimes\Sigma^{-n} B^n(A),\partial_{\theta}).
\end{equation*}
In the cohomology case, our result reads
\begin{equation*}
H^*_{E_n}(A,M) = H_*(\underline{\Hom}_k(\Sigma^{-n}B^n(A),M),\partial_{\theta}),
\end{equation*}
for a twisted cochain complex $B_{[n]}^*(A,M) = (\underline{\Hom}_k(\Sigma^{-n}B^n(A),M),\partial_{\theta})$ which we obtain
by replacing the tensor product of the homological construction $M\otimes B^n(A)$
by a differential graded module of homomorphisms $\underline{\Hom}_k(B^n(A),M)$.
These results hold for all $1\leq n\leq\infty$, including $n=\infty$ for a suitable infinite bar complex $\Sigma^{-\infty}B^{\infty}(A)$
which we define as a colimit of the finitely iterated bar constructions $\Sigma^{-n}B^n(A)$, $n = 1,2,\dots$.

Let us emphasize that we use the plain Eilenberg-MacLane iterated bar complexes in our approach.
The chain complexes which we define in this paper therefore give a small algebraic counterpart of the constructions
of factorization homology theory~\cite{Fra13}, which are based on higher category theory models,
and of the constructions of Pirashvili's higher Hochschild homology theory~\cite{Pir00},
which is defined in terms of functors
on simplicial sets.
But we do not use this analogy further in this paper. To get our result, we mainly rely on the theory of modules over operads~\cite{Fre09} and on the methods of~\cite{Fre11}
which we extend to appropriately address the non-trivial coefficient setting.
The small explicit complexes for calculating $E_n$-homology and -cohomology we exhibit in this article have been employed by the second author in \cite{Z} to show that $E_n$-homology and -cohomology can be interpreted as functor homology.

\subsection*{Plan}
We define the chain complexes $B_*^{[n]}$ and $B^*_{[n]}$ and we state the main theorems
in the first section
of the paper.
We explain the correspondence between these chain complexes
and modules over operads
in the second section,
and we prove that these chain complexes effectively compute the $E_n$-homology and the $E_n$-cohomology of commutative algebras
in the third section of the paper.

\subsection*{Conventions}
We fix a unital commutative ring $k$. We work in the category $\dgmod$ of lower $\mathbb{Z}$-graded differential graded $k$-modules, with the differential lowering the degrees.
For a differential graded module $M$ and $m \in M$ we denote the corresponding element in the suspension $\Sigma M$ by $sm$.
For an associative nonunital differential graded algebra $B$ we denote by $B_+$ the unital algebra obtained by adjoining a unit to $B$.

For $\underline e$ a finite set, we denote by $\vert e \vert$ its arity.
For $s\geq 0$ we set $\underline s = \lbrace 1,\dots,s\rbrace$ and $[s] = \lbrace 0,\dots,s \rbrace$.
By $k<\underline e>$ we denote the free $k$-module with basis $\underline e$.

We denote the symmetric operad encoding associative algebras in differential graded $k$-modules by $\As$
and the symmetric operad encoding commutative algebras in differential graded $k$-modules by $\Com$.
We choose to work with the differential graded version of the Barratt-Eccles operad $E$.
Recall that $E$ is a $\Sigma_*$-cofibrant $E_{\infty}$-operad which admits a filtration
$E_1 \subset E_2 \subset \dots \subset E_n \subset E_{n+1} \subset\dots \subset E$,
such that $E_n$ is a $\Sigma_*$-cofibrant $E_n$-operad, see \cite{Be97}.
We also set $E_{\infty} = E$ for $n=\infty$.
Note that $E$ admits a morphism $E\ra \Com$, hence by restriction of structure we can view every differential graded commutative algebra
as an $E_n$-algebra.
We mostly deal with non-unital algebras in what follows.
We therefore assume $E(0) =0$ when we deal with the Barratt-Eccles operad $E$. We similarly set $\As(0) = 0$
for the associative operad $\As$, and $\Com(0) = 0$
for the commutative operad $\Com$.

For the necessary background on symmetric sequences, operads, algebras over an operad and right modules over an operad we refer the reader to \cite{Fre09}. We denote the plethysm for symmetric sequences by $\circ$.
We will frequently switch between considering $\Sigma_*$-modules (with the symmetric groups acting on the right) and contravariant functors defined on the category $\Bij$ of finite sets and bijections as explained for example in \cite[0.2]{Fre11}.

Let $K\circ\mcP$ be the free right $\mcP$-module over the operad $\mcP$ generated by the $\Sigma_*$-module $K$. Given a morphism $f \colon K \ra R$ of $\Sigma_*$-modules with target a right $\mcP$-module, we denote by $\partial_f\colon K\circ\mcP \ra R$ the induced morphism of right $\mcP$-modules.

\section{The results}

We briefly explain our results in this section.

Since $E_n$ is $\Sigma_*$-cofibrant, the category of $E_n$-algebras forms a semi-model category (see \cite[ch.12]{Fre09}). In particular there is a notion of cofibrant replacement for $E_n$-algebras, and we can define their homology and cohomology as derived functors
\begin{equation*}
H^*_{E_n}(A;M) = H_*(\Der_{E_n}(Q_A,M)) \quad \text{and} \quad H_*^{E_n}(A;M) =H_*( M \otimes_{U_{E_n}(Q_A)} \Omega^1_{E_n}(Q_A))
\end{equation*}
for a cofibrant replacement $Q_A$ of the $E_n$-algebra $A$ and a representation $M$ of $A$.
Here for an operad $\mcP$ and a $\mcP$-algebra $A$, a representation $M$ of $A$ can be thought of as generalizing the notion
of an $A$-module.
By $U_{\mcP}(A)$ we denote the enveloping algebra of $A$. Representations of $A$ correspond to left modules over $U_{\mcP}(A)$.
We denote by $\Der_{\mcP}(A;M)$ the differential graded module of $\mcP$-derivations from $A$ to $M$,
and $\Omega^1_{\mcP}(A)$ is the associated module of K\"ahler differentials, satisfying
\begin{equation*}
\Der_{\mcP}(A,M) \cong \underline{\Hom}_{U_{\mcP}(A)}(\Omega^1_{\mcP}(A),M),
\end{equation*}
where $\underline{\Hom}_{U_{\mcP}(A)}(-,-)$ denotes the hom-object
enriched in differential graded modules
of the category of differential graded left modules over $U_{\mcP}(A)$.

For $\mcP=\Com$ we retrieve the usual notions of symmetric bimodules, derivations and K\"ahler differentials.
The universal enveloping algebra is given by
\begin{equation*}
U_{\Com}(A) = A_+.
\end{equation*}

All of these objects are appropriately natural, in particular they are natural in $\mcP$. Hence any symmetric bimodule over a commutative algebra $A$ is also a representation of $A$ as an $E_n$-algebra. We refer the reader to \cite[ch.4, ch.10, ch.13]{Fre09} for further background on these constructions.

In \cite{Fre11} the first author proves that the iterated bar construction can be extended to $E_n$-algebras, and that up to a suspension this complex calculates the $E_n$-homology of a given $E_n$-algebra with coefficients in $k$. In this article, we prove a corresponding result for $E_n$-homology of commutative algebras with coefficients in a symmetric bimodule.

The bar construction is usually defined for augmented unital associative algebras, but we prefer to deal with nonunital algebras.
To retrieve the usual construction, we simply use that a nonunital algebra $A$
is equivalent to the augmentation ideal of the algebra $A_+$
which we define by adding a unit to $A$.

The bar construction $B(A)=(\overline{T}{}^c(\Sigma\overline{A}),  \partial_s)$ of a differential graded associative algebra $A$
is given by the reduced tensor coalgebra $\overline{T}{}^c(-)$ on the suspension of $A$
together with a twisting differential $\partial_s$ which is determined by the multiplication of $A$.
If $A$ is commutative, then the bar construction $B(A)$
also becomes a commutative differential graded algebra with the shuffle product as product,
and we can iteratively define $B^n(A)$. We review this construction in more detail in the next section.

The iterated bar construction $B^n(A)$ is spanned by planar fully grown $n$-level trees whose leaves are labeled by a tensor in $A$.
For $1 \leq n < \infty$, we formally define a planar fully grown $n$-level tree $t$ (an $n$-level tree for short)
as a sequence of order-preserving surjections
\begin{equation*}
t=\tree.
\end{equation*}
The elements in $[r_i]$ are the vertices at level $i$ of the tree. The elements of $[r_n]$ are the leaves of $t$.
The labeled trees which we consider in the definition of the bar construction are pairs $(t,\pi)$,
where $t$ is an $n$-level tree and $\pi = a_0\otimes\dots\otimes a_{r_n}\in A^{\otimes [r_n]}$
is a tensor whose factors are indexed by the leaves $[r_n]$.
We adopt the notation $t(a_0,\dots,a_{r_n})$ for the element of the $n$-fold bar complex defined by such a pair $(t,\pi)$.

In what follows, we also consider trees equipped with a labeling by a finite set $\underline e$.
In this case, the labeling is a bijection
\begin{equation*}
\sigma: \underline e\xrightarrow{\cong}[r_n].
\end{equation*}

We refer the reader to \cite{Fre11} and \cite{FrApp} for a description of the differential on $B^n(A)$.
Before we define the complex calculating $E_n$-homology we need to introduce a notation concerning trees.

\begin{defn}
Let $t=\tree$ be an $n$-level tree. For $s \in [r_n]$ such that $s$ is not the only element in the corresponding $1$-fiber of $t$ containing $s$ we let $t\setminus s$ be the tree
\begin{equation*}t\setminus s = \xymatrix{[r_n-1] \ar[r]^-{f_n'} & [r_{n-1}] \ar[r]^-{f_{n-1}} & \cdots\ar[r]^-{f_2}  & [r_1]}\end{equation*}
with
\begin{equation*}f_n'(x) = \begin{cases}
f_n(x), & x< s,\\
f_n(x+1), & x\geq s.
\end{cases}\end{equation*}
To a tree $(t,\sigma)$ labeled by a finite set $\underline{e}$, we also associate the tree $(t\setminus s,\sigma'_{\sigma^{-1}(s)})$
equipped with the labeling $\sigma'_{\sigma^{-1}(s)} \colon  \underline e \setminus \lbrace \sigma^{-1}(s) \rbrace \ra [r_n-1]$
such that:
\begin{equation*}\sigma'_{\sigma^{-1}(s)}(e) = \begin{cases}
\sigma(e), & \sigma(e)< s,\\
\sigma(e)+1, & \sigma(e)> s.
\end{cases}
\end{equation*}
\end{defn}

\begin{defn}\label{defn:twistingdifferential}
Let $A$ be a differential graded commutative algebra. We define a twisting differential $\partial_{\theta}$ by setting
\begin{eqnarray*}
\partial_{\theta}(u\otimes t(a_0,\dots,a_{r_n})) & = &  \sum_{\substack{0\leq l \leq r_{n-1},\\ \vert f_n^{-1}(l)\vert >1,\\ x=\min f_n^{-1}(l)}} (-1)^{s_{n,x}-1+\vert a_x \vert( \vert a_0 \vert +\dots+\vert a_{x-1} \vert)}  ua_x  \otimes (t \setminus x)(a_0,\dots,\hat{a_x},\dots,a_{r_n})\\
&+& \sum_{\substack{0 \leq l \leq r_{n-1}\\ \vert f_{n}^{-1}(l)\vert >1,\\ y=\max f_n^{-1}(l)}} (-1)^{s_{n,y}+ \vert a_y \vert(\vert u \vert +\vert a_0 \vert +\dots+\vert a_{y-1} \vert)}  a_y u  \otimes (t \setminus y)(a_0,\dots,\hat{a_y},\dots,a_{r_n})
\end{eqnarray*}
for a tree $t= \tree$ labeled by $a_0,\dots,a_{r_n} \in  A$ and $u\in A_+$. We omit the definition of the signs $s_{n,a}$ at this point, the definition will be given in \ref{defn:theta}.
\end{defn}

For example, the element represented by the decorated $2$-level tree
\begin{equation*}
\setlength{\unitlength}{0.18cm}
\begin{picture}(18, 12)
\put(0.5, 5){$u$}
\put(2.5, 5){$\otimes$}
\put(8, 5){\line(1, -1){4}}
\put(16, 5){\line(-1, -1){4}}
\put(5, 9){\line(3, -4){3}}
\put(8, 9){\line(0, -1){4}}
\put(11, 9){\line(-3, -4){3}}
\put(4.5, 10){$a_0 $}
\put(7.5, 10){$a_1$}
\put(10.5, 10){$a_2$}
\put(14, 9){\line(1, -2){2}}
\put(18, 9){\line(-1, -2){2}}
\put(13.5, 10){$a_3 $}
\put(17.5, 10){$a_4$}
\end{picture}
\end{equation*}
is mapped by  $\partial_{\theta}$ to
\begin{equation*}
\setlength{\unitlength}{0.18cm}
\begin{picture}(80, 12)
\put(0, 5){$-ua_0$}
\put(6, 5){$\otimes$}
\put(10.5, 5){\line(1, -1){4}}
\put(18.5, 5){\line(-1, -1){4}}
\put(8.5, 9){\line(1, -2){2}}
\put(12.5, 9){\line(-1, -2){2}}
\put(8, 10){$a_1 $}
\put(12, 10){$a_2$}
\put(16.5, 9){\line(1, -2){2}}
\put(20.5, 9){\line(-1, -2){2}}
\put(16, 10){$a_3 $}
\put(20, 10){$a_4$}

\put(23.5,5){$-$}

\put(26, 5){$(-1)^{\vert a_3 \vert (\vert a_0\vert + \vert a_1\vert + \vert a_2 \vert)} ua_3$}
\put(49.5, 5){$\otimes$}
\put(55, 5){\line(1, -1){4}}
\put(63, 5){\line(-1, -1){4}}
\put(52, 9){\line(3, -4){3}}
\put(55, 9){\line(0, -1){4}}
\put(58, 9){\line(-3, -4){3}}
\put(51.5, 10){$a_0 $}
\put(54.5, 10){$a_1$}
\put(57.5, 10){$a_2$}
\put(63, 9){\line(0, -1){4}}
\put(62.5, 10){$a_4 $}
\end{picture}
\end{equation*}

\begin{equation*}
\setlength{\unitlength}{0.18cm}
\begin{picture}(80, 12)
\put(-3, 5){$+$}
\put(0, 5){$(-1)^{\vert a_2 \vert (\vert u \vert + \vert a_0 \vert + \vert a_1\vert)}a_2u$}
\put(22.5, 5){$\otimes$}
\put(27, 5){\line(1, -1){4}}
\put(35, 5){\line(-1, -1){4}}
\put(25, 9){\line(1, -2){2}}
\put(29, 9){\line(-1, -2){2}}
\put(24.5, 10){$a_0 $}
\put(28.5, 10){$a_1$}
\put(33, 9){\line(1, -2){2}}
\put(37, 9){\line(-1, -2){2}}
\put(32.5, 10){$a_3 $}
\put(36.5, 10){$a_4$}

\put(38.5,5){$-$}

\put(41.5, 5){$(-1)^{\vert a_4 \vert (\vert u \vert + \vert a_0\vert + \dots+ \vert a_3\vert)} a_4u$}
\put(67, 5){$\otimes$}
\put(72.5, 5){\line(1, -1){4}}
\put(80.5, 5){\line(-1, -1){4}}
\put(69.5, 9){\line(3, -4){3}}
\put(72.5, 9){\line(0, -1){4}}
\put(75.5, 9){\line(-3, -4){3}}
\put(69, 10){$a_0 $}
\put(72, 10){$a_1$}
\put(75, 10){$a_2$}
\put(80.5, 9){\line(0, -1){4}}
\put(80, 10){$a_3 $}
\end{picture}
\end{equation*}

For a symmetric $A$-bimodule $M$, we set:
\begin{align*}
B^{[n]}_*(A,M) & = M\otimes_{A_+}(A_+\otimes \Sigma^{-n}B^n(A),\partial_{\theta})
\intertext{and}
B_{[n]}^*(A,M) & = \underline{\Hom}_{A_+}((A_+\otimes\Sigma^{-n} B^n(A),\partial_{\theta}), M).
\end{align*}
In this construction, we just use that the algebra $U_{\Com}(A) = A_+$ is identified with a symmetric bimodule over itself.

This article is dedicated to proving the following results.

\begin{thm}\label{mainresult}
Let $1 \leq n \leq \infty$. Let $A$ be a commutative differential graded algebra and let $M$ be a symmetric $A$-bimodule in $\dgmod$.
Provided that $A$ is cofibrant in $\dgmod$,
we have the identities
\begin{equation*}
H_*^{E_n}(A; M) = H_*(B^{[n]}_*(A,M))\quad\text{and}
\quad H^*_{E_n}(A; M) = H_*(B_{[n]}^*(A,M)).
\end{equation*}
\end{thm}

We complete the proof of this theorem in Section~\ref{section:ModuleToHomology}.


\section{From iterated bar complexes to modules over operads}

\subsection*{Reminder on the trivial coefficient case}

The (unreduced) bar construction was originally defined by Eilenberg-MacLane for augmented algebras in $\dgmod$ in \cite[II.7]{EM53}. We recall from \cite[1.6]{Fre11} the definition of the nonunital reduced bar construction in the context of right modules over an operad, which we will be working in.

Let $\mcP$ be an operad in differential graded modules and let $R$ be an $\As$-algebra in right $\mcP$-modules.
The bar construction associated to $R$ is a twisted right $\mcP$-module such that $BR = (\overline{T}^c(\Sigma R), \partial_s)$,
where we use the tensor product of right $\mcP$-modules
to form the tensor coalgebra
\begin{equation*}
\overline{T}^c(\Sigma R) = \bigoplus_{l\geq 1} (\Sigma R)^{\otimes l}.
\end{equation*}
For a finite set $\underline e$ the object $(\Sigma R)^{\otimes l}(\underline e)$
is spanned by tensors
$sr_1 \otimes \dots \otimes sr_l \in \Sigma R (\underline{e_1}) \otimes \dots \otimes \Sigma R(\underline{e_l})$
such that $\underline e = \underline{e_1} \sqcup \dots \sqcup \underline{e_l}$.
The twist $\partial_s$ is defined by
\begin{equation*}
\partial_s(sr_1 \otimes\dots\otimes sr_l) = \sum_{i=1}^{l-1} (-1)^{i+\vert r_1 \vert +\dots+ \vert r_i \vert } sr_1 \otimes \dots \otimes s\gamma(\id_{\underline 2};r_i,r_{i+1}) \otimes \dots \otimes sr_l
\end{equation*}
where $\id_{\underline{2}} \in \As(2)$ and $\gamma$ refers to the action of $\As$ on $R$.

If $R$ is commutative, \ie if the action of $\As$ on $R$ factors through $\Com$, then $BR$ is a $\Com$-algebra in right $\mcP$-modules. For $sr_1 \otimes \dots \otimes sr_p \in \Sigma R(\underline{e_1}) \otimes \dots \otimes \Sigma R(\underline{e_p}), sr_{p+1} \otimes \dots \otimes sr_{p+q} \in \Sigma R(\underline{e_{p+1}}) \otimes \dots \otimes \Sigma R(\underline{e_{p+q}})$
the multiplication is given by the shuffle product
\begin{equation*}
\sh( sr_1 \otimes \dots \otimes sr_p , sr_{p+1} \otimes \dots \otimes sr_{p+q})
= \sum_{\sigma \in \sh(p,q)} (-1)^{\epsilon} sr_{\sigma^{-1} (1)} \otimes \dots \otimes sr_{\sigma^{-1}(p+q)},
\end{equation*}
where $\sh(p,q) \subset \Sigma_{p+q}$ denotes the set of shuffles of $\lbrace 1,\dots,p \rbrace$ with $\lbrace p+1,\dots,p+q\rbrace$. The sign $\epsilon$ is the graded signature of the shuffle $\sigma$, \ie it is determined  by picking up a factor $(-1)^{(\vert r_i \vert +1) (\vert r_j \vert +1)}$ whenever $i<j$ and $\sigma (i)> \sigma(j)$.

In particular, if $R$ is commutative we can iterate the construction and define an $n$-fold bar complex $B^n(R)$.
Applying this to $\Com$ itself, regarded as a $\Com$-algebra in right $\Com$-modules, we define the commutative algebra $B^n_{\Com}$ in right $\Com$-modules by
\begin{equation*}
B^n_{\Com} := B^n(\Com).
\end{equation*}
According to \cite[2.7, 2.8]{Fre11} the iterated bar module $B^n_{\Com}$ is a quasifree right $\Com$-module
$B^n_{\Com} = (T^n \circ \Com, \partial_{\gamma})$
with $T^n= (\overline{T}^c \Sigma)^n(I)$ a free $\Sigma_*$-module, which we will discuss in more detail later.
By \cite[2.5]{Fre11} it is possible to lift $\partial_{\gamma}$ to a twisting differential
$\partial_{\epsilon} \colon T^n   \circ E \ra T^n \circ E$
and set $B^n_{E} = (T^n \circ  E, \partial_{\epsilon})$.
For an $E$-algebra $A$ we call
\begin{equation*}
B^n_{E}(A) = B^n_{E} \circ_{E} A
\end{equation*}
the $n$-fold bar complex of $A$.
Using that $E$ is equipped with a cell structure indexed by complete graphs the first author proves in \cite[5.4]{Fre11} that $\partial_{\epsilon}$ restricts to
$\partial_{\epsilon} \colon T^n \circ  E_n \ra T^n  \circ E_n$.
Hence setting $B^n_{E_n} = (T^n \circ E_n, \partial_{\epsilon})$ we can define an $n$-fold bar complex
$B^n_{E_n} (A) = B^n_{E_n} \circ_{E_n} A$
for any $E_n$-algebra $A$.
If $A$ is commutative, \ie if $A$ is an $E$-algebra such that the action of $E$ on $A$ factors as the standard trivial fibration $E \ra \Com$ followed by an action of $\Com$ on $A$, then
$B^n_{E_n}(A) = B^n_E(A) = B^n_{\Com}(A)$
is the usual bar construction.

For any $E$-algebra $A$ in $\dgmod$ or in right modules over an operad there is a suspension morphism
$\sigma \colon \Sigma A \ra B_{E}(A)$,
induced by the inclusion $i \colon \Sigma I \ra T^c \Sigma I$. Hence we can set
\begin{equation*}
\Sigma^{-\infty}B^{\infty}(A) = \colim_n \Sigma^{-n}B^n_E(A).
\end{equation*}
In particular we can define the right $E$-module $\Sigma^{-\infty}B^{\infty}_E$. Since the associated suspension morphism is of the form $\sigma= i\circ E$ this is again a quasifree right $E$-module, generated by the $\Sigma_*$-module
$\Sigma^{-\infty}T^{\infty} = \colim_n \Sigma^{-n}T^n$.
Similar considerations apply to $\Sigma^{-\infty}B^{\infty}_{\Com}$, and we have $\Sigma^{-\infty}B^{\infty}_E(\Com)= \Sigma^{-\infty}B^{\infty}_{\Com}$.

As a quasifree right $E_n$-module in nonnegatively graded chain complexes $B^n_{E_n}$ is a cofibrant right $E_n$-module.
The augmentation of the desuspended $n$-fold bar complex $\epsilon \colon \Sigma^{-n}B^n_{E_n}\ra I$
is defined as the composite
$\Sigma^{-n}B^n_{E_n}   \ra \Sigma^{-n}T^n  I \ra \Sigma^{-n} \Sigma^{n}(I) = I$
with the first map induced by the operad morphism $E_n \ra \Com \ra I$ and the second map an iteration of the projection $\overline{T}^c (\Sigma I )\ra \Sigma I$.
In \cite[8.21, 9.4]{Fre11} the first author shows that $\epsilon$ is a quasiisomorphism, hence $\Sigma^{-n} B^n_{E_n}$ is a cofibrant replacement of $I$. From this one deduces that, for an $E_n$-algebra $A$
which is cofibrant as a differential graded $k$-module, there is an isomorphism
\begin{equation*}
H_*^{E_n}(A;k) \cong H_*(\Sigma^{-n}B^n_{E_n}(A))
\end{equation*}
for $1\leq n \leq \infty$ (see \cite[8.22, 9.5]{Fre11}).


\subsection*{The twisted module modeling iterated bar complexes with coefficients}
\label{sec:theta}

We now define the twist $\partial_{\theta}$ on $M \otimes B^n_{\Com}(A)$ which will incorporate the action of the nonunital commutative algebra $A$ on the symmetric $A$-bimodule $M$. As we will see the general case follows from the case of universal coefficients $M=U_{\Com}(A) = A_+$. We first recall the definition of the universal enveloping algebra associated to $\Com$. Note that this is a special case of the notion of enveloping algebras associated to operads, which we will discuss in Proposition \ref{prop:UnivEnv}.

\begin{defn}
We denote by $U_{\Com}$ the algebra in right $\Com$-modules given  by
$U_{\Com}(i)=\Com(i+1)$, \ie
$U_{\Com}(i) = k$
for all $i\geq 0$ with trivial $\Sigma_i$-action, see \eg \cite[10.2.1]{Fre09}. Equivalently,
\begin{equation*}
U_{\Com}(\underline e) = k
\end{equation*}
for all finite sets $\underline e$. Denote by $\mu_{\underline e}$ the generator of $\Com(\underline e)=k$.  Set $\underline{f}_+= \underline f \sqcup \lbrace + \rbrace$ and
denote the generator $\mu_{\underline{f}_+} \in U_{\Com}(\underline f)=\Com(\underline f \sqcup +)$ by $\mu_{\underline{f}}^U$. Then the multiplication in $U_{\Com}$ is
\begin{equation*}
\mu^U_{\underline e}\cdot  \mu^U_{\underline f} =\mu^U_{\underline e \sqcup \underline f},
\end{equation*}
while the right $\Com$-module structure is given by
\begin{equation*}
\mu^U_{\underline e} \circ_e \mu_{\underline f} = \mu^U_{(\underline e \sqcup \underline f)\setminus \lbrace e \rbrace}
\end{equation*}
for $e \in \underline e$.
\end{defn}

Observe that $A_+ =U_{\Com}\circ_{\Com} A$. Hence
\begin{equation*}
U_{\Com}(A) \otimes B^n(A) = (U_{\Com} \otimes B^n_{\Com}) \circ_{\Com} A
\end{equation*}
and it suffices to define
\begin{equation*}
\partial_{\theta} \colon U_{\Com} \otimes B^n_{\Com}\ra U_{\Com}\otimes B^n_{\Com}.
\end{equation*}

Before giving the definition of this twisting differential, we check that $U_{\Com} \otimes (T^n \circ {\Com})$ is a free $U_{\Com}$-module in right $\Com$-modules, and we examine the generating $\Sigma_*$-module $T^n$ closer.

\begin{defn}
Let $\mcP$ be an operad and $(U,\mu_U, 1_U)$ an associative unital algebra in right $\mcP$-modules with $\mcP$-action $\gamma_U \colon U \circ\mcP \ra U$.
The category $\UModP$ of left $U$-modules in right $\mcP$-modules consists of right $\mcP$-modules $(M,\gamma_M)$ equipped with a left $U$-action $\mu_M\colon U \otimes M \ra M$ which is a morphism of right $\mcP$-modules.
The morphisms of this category are morphisms of right $\mcP$-modules which preserve the $U$-action.
\end{defn}

\begin{proposition}\label{prop:FreeUModP}
Let $M$ be a $\Sigma_*$-module. The free object in $\UModP$ generated by $M$ is $U \otimes (M \circ \mcP)$
with $U$-action given by
\begin{equation*}
\xymatrix{U \otimes U \otimes ( M \circ \mcP) \ar[rr]^-{\mu_U \otimes (M \circ \mcP)} && U \otimes (M \circ \mcP)}
\end{equation*}
and right $\mcP$-module structure defined by
\begin{equation*}
\xymatrix{(U \otimes (M \circ \mcP)) \circ \mcP \ar[r]^-{\cong} & (U\circ \mcP) \otimes (M \circ \mcP \circ \mcP) \ar[rr]^-{\gamma_U \otimes (M \circ \gamma_{\mcP})} && U \otimes (M \circ \mcP)}.
\end{equation*}
\end{proposition}

In the following, for a map $f \colon M \ra N$ from a $\Sigma_*$-module $M$ to $N \in \UModP$ we will denote by
\begin{equation*}\partial_f \colon U \otimes (M \circ\mcP) \ra N\end{equation*}
the associated morphism in $\UModP$. We will use the same notation for morphisms
defined on free right $\mcP$-modules, which version applies will be clear from the context.

Let $1 \leq n < \infty$. By definition, we have $T^n = (\overline{T}^c \Sigma)^n(I)$ and hence, this object has an expansion
such that
\begin{align*}
T^n(\underline e) & = \bigoplus_{\underline e = \underline{e_1} \sqcup \dots \sqcup \underline{e_l}} \Sigma T^{n-1} ( \underline{e_1}) \otimes \dots \otimes \Sigma T^{n-1}(\underline{e_l})
\intertext{and with}
T^1(\underline e) & = (I^{\otimes s})(\underline e) =  \bigoplus_{\sigma \colon \underline e \ra \underline s } k \cdot \sigma
\end{align*}
for $\vert \underline e \vert =s$.
For $n=1$ it is obvious that $T^n(\underline e)$ has generators corresponding to decorated $1$-level trees $(t = [r_1],\sigma: \underline{e}\ra [r_1])$.
For $n>1$, let $sx_1 \otimes \dots \otimes sx_l \in \Sigma T^{n-1} ( \underline{e_1}) \otimes \dots \otimes \Sigma T^{n-1}(\underline{e_l})$. Then the corresponding $n$-level tree is the $n$-level tree with $l$ vertices in level $1$ such that the $i$th vertex is the root of the $n-1$-level tree defined by $x_i$. Hence elements in $T^n(\underline e)$ correspond to $n$-level trees with leaves decorated by $\underline e$. Consequently, as a $k$-module $T^n\circ {\Com}$ is generated by planar fully grown $n$-level trees in $T^n$ with leaves labeled by elements in $\Com$.

\begin{defn}\label{defn:theta}
The morphism
$\partial_{\theta}\colon U_{\Com} \otimes (T^n \circ {\Com}) \ra U_{\Com} \otimes (T^n \circ {\Com})$
in $\UCModMC$ is induced by a map
\begin{equation*}
\theta \colon T^n \ra U_{\Com} \otimes (T^n \circ {\Com})
\end{equation*}
which we define as follows: for $(t,\sigma) \in T^n(\underline e)$ with $t=\tree$
labeled by $\sigma \colon \underline e \ra [r_n]$
we set
\begin{align*}
\theta(t, \sigma)& = \sum_{\substack{0\leq l \leq r_{n-1},\\ \vert f_n^{-1}(l)\vert >1,\\ x=\min f_n^{-1}(l)}} (-1)^{s_{n,x}-1} \mu_{\lbrace \sigma^{-1}(x)\rbrace}^U \otimes (t \setminus x, \sigma'_{\sigma^{-1}(x)})\\
& + \sum_{\substack{0 \leq l \leq r_{n-1}\\ \vert f_{n}^{-1}(l)\vert >1,\\ y=\max f_n^{-1}(l)}} (-1)^{s_{n,y}}\mu_{\lbrace\sigma^{-1}(y)\rbrace}^U \otimes (t \setminus y, \sigma'_{\sigma^{-1}(y)}).
\end{align*}
The sign $(-1)^{s_{n,i}}$ is determined by counting the edges in the tree $t$ from bottom to top and from left to right. Then $s_{n,i}$ is the number assigned to the edge connected to the $i$th leave. These signs can be thought of as originating from the process of switching the map $M \otimes \Sigma A \ra M$ of degree $-1$ past the suspensions in the bar construction.
\end{defn}

A tedious but straightforward calculation yields that $\theta$ is indeed a twisting cochain:

\begin{lemma}
The map $\theta$ defines a twisting cochain on $U_{\Com} \otimes B^n_{\Com}$, {\ie} we have
$(\partial_{\theta} + U_{\Com} \otimes \partial_{\gamma})^2=0$.
\end{lemma}

The suspension morphism $\Sigma B^n_{\Com} \ra B^{n+1}_{\Com}$ commutes with the twists $\partial_{\theta}$ defined on $B^n_{\Com}$ and $B^{n+1}_{\Com}$, hence we can extend the above definition to $n=\infty$.

\begin{proposition}\label{defn:thetaForInfty}
The twist $\partial_{\theta}$ extends to a twisting differential on
$U_{\Com} \otimes \Sigma^{-\infty}B^{\infty}_{\Com}$.
\end{proposition}

Let $\eta_{U} \colon k \ra U_{\Com}(0)$ denote the unit map of $U_{\Com}$. We want to lift
$ \partial_{\theta} + U_{\Com}\otimes \partial_{\gamma} = \partial_{\theta + \eta_U \otimes \gamma}$
 to $U_{\Com}\otimes B^n_E$. To achieve this we mimick the construction that is used in \cite[2.4]{Fre11} to lift $\partial_{\gamma}$ to $B^n_E$.
The following proposition extends \cite[2.5]{Fre11}.

\begin{proposition}\label{prop:lifting}
Let $\mcR,\mcS$ be operads equipped with differentials $d_{\mcR}$ and $d_{\mcS}$ and let $U$ be an algebra in right $\mcR$-modules.
Suppose there are maps
\begin{gather*}
\xymatrix{ \ar@(ul,dl)[]_-{\nu} \mcR \ar@<1ex>[r]^-{\psi} & \ar@<1ex>[l]^-{\iota} \mcS }
\intertext{such that $\psi$ is a morphism of operads, $\iota$ is a chain map and}
\psi \iota=\id, \quad  d_{\mcR} \nu - \nu d_{\mcR} = \id - \iota \psi \quad \text{and} \quad \psi \nu =0.
\intertext{Let $K=G \otimes \Sigma_*$ be a free $\Sigma_*$-module and}
\beta \colon G  \ra U \otimes  (K \circ \mcS)
\intertext{a twisting cochain which additionally satisfies}
d_{U\otimes (K\circ\mcS)} \beta + \beta d_G=0 \quad \text{ and } \quad \partial_{\beta } \beta =0
\end{gather*}
with $d_{U \otimes (K\circ\mcS)}$ and $d_G$ denoting the differentials on $U\otimes (K\circ\mcS)$ and $G$.

If $K, U$ and $\mcS$ are nonnegatively graded there exists a twisting cochain $\alpha \colon K \ra U \otimes (K\circ\mcR)$ such that
\begin{equation*}
\xymatrix{ U \otimes (K\circ\mcR) \ar[d]^-{U \otimes (K\circ\psi)} \ar[r]^-{\partial_{\alpha}} &
U \otimes (K\circ\mcR)\ar[d]^-{U \otimes (K\circ\psi)}\\
U \otimes (K\circ\mcS) \ar[r]^-{\partial_{\beta}} & U \otimes (K\circ\mcS). }
\end{equation*}
\end{proposition}

\begin{proof}
Extend the homotopy $\nu \colon \mcR \ra \mcR$ to $\hat \nu^{(l)}\colon \mcR^{\otimes l} \ra \mcR^{\otimes l}$ by setting
\begin{equation*}
\hat{\nu}^{(l)} = \sum_{i=1}^l (-1)^{i-1} (\iota \psi)^{\otimes i-1}\otimes \nu \otimes \id_{\mcR}^{\otimes l-i}
\end{equation*}
and extend $\iota$ to $\hat\iota^{(l)} \colon \mcS^{\otimes l}\ra \mcR^{\otimes l}$
by setting
\begin{equation*}
\hat{\iota}^{(l)} = \iota^{\otimes l}.
\end{equation*}
We have $K\circ\mcR = \bigoplus_{i\geq 0} G(i) \otimes \mcR^{\otimes i}$ since $K$ is $\Sigma_*$-free,
and similarly  $K\circ\mcS = \bigoplus_{i\geq 0} G(i) \otimes \mcS^{\otimes i}$.
We can then define
\begin{equation*}
\tilde \nu \colon U \otimes (K\circ\mcR) \ra U \otimes (K\circ\mcR) \quad \text{ and } \quad \tilde \iota \colon U \otimes (K\circ\mcS) \ra U\otimes (K\circ\mcR)
\end{equation*}
as $\tilde \nu =\id_U \otimes \bigoplus_{l \geq 0} G(l) \otimes \hat{\nu}^{(l)}$ and $\tilde \iota = \id_U \otimes \bigoplus_{l \geq 0} G(l) \otimes \hat{\iota}^{(l)}$.
Note that with this definition, we get:
\begin{equation*}
(U \otimes (K\circ \psi)) \tilde \iota = \id , \quad \delta(\tilde \nu) = \id_{U \otimes (K\circ\mcR)} - \tilde\iota (U\otimes (K\circ \psi)) \quad \text{ and } \quad (U \otimes (K\circ \psi)) \tilde \nu=0.
\end{equation*}
We define $\alpha \colon G \ra U \otimes (K\circ\mcR)$ by setting $\alpha_0= \tilde \iota  \beta,$
\begin{equation*}
\alpha_m = \sum_{a+b=m-1}\tilde \nu \partial_{\alpha_a} \alpha_b
\quad\text{ and }\quad
\alpha = \sum_{m\geq 0}\alpha_m.
\end{equation*}

Observe that $\beta\colon K \ra U \otimes (K\circ\mcS)$ lowers the degree in $K$ by at least $1$, hence $\alpha_m$ lowers the degree in $K$ by $m+1$. Since $K$ is bounded below, $\alpha$ is well defined.
Then for $m\geq 1$
\begin{equation*}
(U \otimes (K\circ\psi)) \alpha_m = (U \otimes (K\circ\psi)) \sum_{a+b=m-1}\tilde \nu \partial_{\alpha_a} \alpha_b  = 0,
\end{equation*}
while  $(U \otimes (K\circ\psi)) \alpha_0 = \beta,$
hence
\begin{equation*}
(U \otimes (K\circ\psi))  \alpha = \beta
\end{equation*}
and the diagram above commutes.
To show that $\partial_{\alpha}$ is indeed a twisting differential one shows by induction that
\begin{equation*}
\delta( \alpha_m) = \sum_{a+b=m-1}\partial_{\alpha_a} \alpha_b,
\end{equation*}
which then yields the claim for $\alpha$.
\end{proof}

Recall that the Barratt-Eccles operad has an extension to finite sets and bijections given by
\begin{equation*}
\mcE(\underline e)_l = k<\Bij( \underline e,\underline r)^{l+1}>.
\end{equation*}
In our constructions we will need to choose a distinguished element in $\Bij(\underline e,\underline r)$ for all $\underline e$, corresponding to an ordering of $\underline e$.
We fix a choice of a family $(\tau_{\underline e})_{\underline e \in \Bij}$ such that for $\underline e \subset \mathbb{N}_0$ the element $\tau_{\underline e}$ corresponds to the canonical order.

\begin{proposition}\cite[1.4]{Fre11}\label{prop:mcEComHomotopyRetract}
The trivial fibration from the Barratt-Eccles operad $E$ to the commutative operad $\Com$
\begin{equation*}
\psi \colon E \ra \Com,
\end{equation*}
which we explicitly define by
\begin{equation*}
E(\underline e)_l \ni (\sigma_0,\dots,\sigma_l) \mapsto \begin{cases}
\mu_{\underline e}, & l=0,\\
0, & l \neq 0,
\end{cases}
\end{equation*}
admits a section $\iota \colon \Com \ra E$. This section is a morphism of arity-graded differential graded modules
defined by
\begin{equation*}
\Com(\underline e) \ni \mu_{\underline e} \mapsto \tau_{\underline{e}},
\end{equation*}
In addition, we have a $k$-linear homotopy $\nu \colon E \ra E$ between $\iota \psi$ and $\id_{E}$,
given by the mapping
\begin{equation*}
E(\underline e)_l \ni (\sigma_0,\dots,\sigma_l) \mapsto (\sigma_0,\dots,\sigma_l, \tau_{\underline{e}}),
\end{equation*}
and such that $\psi \nu =0$.
(Note that this is not a homotopy retract of operads since $\iota$ does not even preserve the action of bijections.)
\end{proposition}

Hence we can lift $\partial_{\theta}+ U_{\Com}\otimes \partial_{\gamma}$ as desired:

\begin{proposition}
For $1\leq n \leq \infty$, there is a map $\lambda \colon \Sigma^{-n}T^n \ra U_{\Com} \otimes  (\Sigma^{-n}T^n\circ E) $ which induces a twisting differential $\partial_{\lambda}$ such that
\begin{equation*}
\xymatrix{ U_{\Com} \otimes  (\Sigma^{-n}T^n \circ E) \ar[rr]^-{\partial_{\lambda}} \ar[d]_-{U_{\Com} \otimes (\Sigma^{-n}T^n \circ \psi)} && U_{\Com} \otimes  (\Sigma^{-n}T^n \circ E)\ar[d]^-{U_{\Com} \otimes (\Sigma^{-n}T^n \circ \psi)}\\
U_{\Com} \otimes (\Sigma^{-n}T^n \circ {\Com} )\ar[rr]^-{\partial_{\theta}+ U_{\Com}\otimes \partial_{\gamma}} && U_{\Com} \otimes  (\Sigma^{-n}T^n\circ{\Com}) }
\end{equation*}
commutes.
\end{proposition}


\subsection*{Reminder on the complete graph operad}

As in \cite{Fre11} we will use that $E$ is equipped with a cell structure indexed by complete graphs to prove that $\partial_{\lambda}$ restricts to $U_{\Com} \otimes  (\Sigma^{-n} T^n \circ E_n)$ for $1\leq n < \infty$.
In this subsection we revisit the relevant definitions and results concerning this cell structure. The complete graph operad and its relation to $E_n$-operads has been discussed by Berger in \cite{Be96}. We assume that $n<\infty$ in this and the next subsection.

\begin{defn}
Let $\underline{e}$ be a finite set with $r$ elements. A complete graph $\kappa = (\sigma, \mu)$ on $\underline e$ consists of an ordering $\sigma \colon \underline e \ra \lbrace 1,\dots,r \rbrace$ together with a symmetric matrix  $\mu=(\mu_{ef})_{e,f \in \underline{e}}$ of elements $\mu_{ef} \in \mathbb{N}_0$ with all diagonal entries $0$. We think of $\sigma$ as a globally coherent orientation of the edges and of $\mu_{ef}$ as the weight of the edge connecting $e$ and $f$.
\end{defn}

\begin{example}
The complete graph
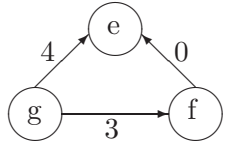
\begin{equation*}\begin{picture}(100, 60)
\put(20,20){\circle{20}}
\put(80,20){\circle{20}}
\put(50,52){\circle{20}}

\put(30,20){\vector(1,0){40}}
\put(22,40){4}
\put(20,30){\vector(1,1){20}}
\put(72,40){0}
\put(80,30){\vector(-1,1){20}}
\put(46,11){3}

\put(17,18){g}
\put(77,18){f}
\put(47,50){e}
\end{picture}\end{equation*}
on the set $\lbrace e,f,g \rbrace$ corresponds to
\begin{equation*}\sigma\colon  \lbrace e,f,g \rbrace \ra \lbrace 1,2,3 \rbrace, \quad \sigma(e)=3, \sigma(f)= 2, \sigma(g) =1\end{equation*}
and $\mu_{ef}=0, \mu_{eg}=4, \mu_{fg}=3.$
\end{example}

For $\sigma$ as above and $e,f \in \underline e$ let $\sigma_{ef} = \id \in \Sigma_2$ if $\sigma(e) < \sigma(f)$. Otherwise, define $\sigma_{ef}$ to be the transposition in $\Sigma_2$.

\begin{defn}
The set of complete graphs on $\underline{e}$ is partially ordered if we set
\begin{equation*}(\sigma, \mu) \leq (\sigma' , \mu')\end{equation*}
whenever for all $e,f \in \underline{e}$ either $\mu_{ef} < \mu'_{ef}$ or $(\sigma_{ef},\mu_{ef}) = (\sigma'_{ef}, \mu'_{ef})$.
The poset of complete graphs on $\underline e$ is denoted by $\mcK(\underline e)$.
\end{defn}

\begin{proposition}\cite{Be96}
The collection $\mcK=(\mcK(\underline e))_{\underline e}$ of complete graphs forms an operad in posets: A bijection $\omega \colon \underline e \ra \underline e'$  acts by relabeling the vertices. The partial composition
\begin{equation*}
\circ_{e} \colon \mcK(\underline e) \times \mcK(\underline f) \ra \mcK(\underline e \sqcup \underline f \setminus \lbrace e \rbrace)
\end{equation*}
is given by substituting the vertex $e$ in a complete graph $(\sigma,\mu)\in \mcK(\underline e)$ by the complete graph $(\tau, \nu) \in \mcK(\underline f)$, \ie by inserting $(\tau, \nu)$ at the position of $e$, orienting the edges between $g \in \underline{e} \setminus \lbrace e \rbrace$ and $g'\in \underline{f}$ like the edge between $g$ and $e$ and giving them the weight $\mu_{g e}$.
\end{proposition}

\begin{defn}\cite[3.4,3.8]{Fre11}
Let $(\mcP(\underline e))_{\underline e}$ be a collection of functors $\mcK(\underline e) \ra \dgmod$. Such a collection is called a $\mcK$-operad if for all finite sets $\underline e$ and all bijections $\omega \colon \underline e' \ra \underline e$ there is a natural transformation $\mcP(\underline e) \ra \mcP(\underline e')$ with components
\begin{equation*}
\mcP_{\kappa} \ra \mcP_{ \kappa .\omega}
\end{equation*}
for $\kappa \in \mcK(\underline e)$,
as well as partial composition products given by natural transformations
\begin{equation*}
\circ_{e} \colon \mcP(\underline e) \otimes \mcP(\underline f) \ra \mcP(\underline{e} \setminus \lbrace e \rbrace \sqcup \underline{f})
\end{equation*}
for $e \in \underline e$ with components
\begin{equation*}
\circ_{e} \colon \mcP_{\kappa} \otimes \mcP_{\kappa'} \ra \mcP_{\kappa \circ_{e} \kappa'},
\end{equation*}
satisfying suitable associativity, unitality and equivariance conditions. A morphism $\mcP \ra \mcP'$ consists of natural transformations with components $\mcP_{\kappa} \ra \mcP'_{\kappa}$ commuting with composition and the action of bijections.

Similarly, a right $\mcK$-module $R$ over a $\mcK$-operad consists of collections $(R(\underline e))_{\underline e}$ together with natural transformations
$R_{\underline e} \ra R_{\underline e'}$
for all bijections $\omega \colon \underline e' \ra \underline e$ with components
\begin{equation*}
R_{\kappa} \ra R_{ \kappa. \omega},
\end{equation*}
as well as natural transformations
\begin{equation*}
\circ_{e} \colon \mcP(\underline e) \otimes R(\underline f) \ra R(\underline{e} \setminus \lbrace e \rbrace \sqcup \underline{f})
\end{equation*}
with components
\begin{equation*}
\circ_e \colon R_{\kappa} \otimes \mcP_{\kappa'} \ra R_{\kappa \circ_{e} \kappa'}
\end{equation*}
for $e \in \underline e$  satisfying again suitable associativity, unitality and equivariance conditions.
\end{defn}

\begin{proposition}\cite[3.4,3.8]{Fre11}
There is an adjunction
\begin{equation*}
\xymatrix{ \colim \colon \mcO_{\dgmod} \ar@<.5ex>[r] & \ar@<.5ex>[l] \KO \colon \const }
\end{equation*}
between the category $\mcO_{\dgmod}$ of operads in differential graded $k$-modules and the category $\KO$ of $\mcK$-operads defined as follows: For a given $\mcK$-operad $\mcP$ let
\begin{equation*}
(\colim \mcP)(\underline{e}) = \colim_{\kappa \in \mcK(\underline e)} \mcP_{\kappa},
\end{equation*}
while for an ordinary operad $\mcQ$ we set
\begin{equation*}
\const(\mcQ)_{\kappa} = \mcQ(\underline e)
\end{equation*}
for $\kappa \in \mcK(\underline e)$. We say that an operad $\mcQ$ has a $\mcK$-structure if $\mcQ = \colim \mcP$.
A similar adjunction exists between the category of right $\mcK$-modules over a $\mcK$-operad $\mcP$ and the category of right modules over $\colim \mcP$, and we call right modules of the form $\colim R$ right $\colim \mcP$-modules with $\mcK$-structure.
\end{proposition}

\begin{example} \cite[3.5]{Fre11}
Besides considering constant $\mcK$-operads the main example we are interested in is the Barratt-Eccles operad $E$. Concretely, for a finite set $\underline e$ with $r$ elements and $\kappa=(\sigma,\mu) \in \mcK(\underline e)$ an element $(\omega_0,\dots,\omega_l) \in \Bij(\underline e, \underline r)$ is in $\mcE_{\kappa}$
if for all $e,f \in \underline e$ the sequence
\begin{equation*}
((\omega_0)_{ef},\dots,(\omega_l)_{ef})
\end{equation*}
has either less than $\mu_{ef}$ variations or has exactly $\mu_{ef}$ variations and $(\omega_l)_{ef}=\sigma_{ef}$.
Observe that in
\begin{equation*}
\xymatrix{ \ar@(ul,dl)[]_-{\nu} E \ar@<1ex>[r]^-{\psi} & \ar@<1ex>[l]^-{\iota} \Com, }
\end{equation*}
the map $\psi$ respects the $\mcK$-structures on $E$ and $\Com$ and that for  $\kappa= (\tau_{e},\mu)$ we have
\begin{equation*}
\iota(\Com_{\kappa}) \subset E_{\kappa} \quad \text{and} \quad \nu(E_{\kappa}) \subset E_{\kappa}.
\end{equation*}
\end{example}

\begin{defn}
Let $\mcK_n(r)$ be the poset of complete graphs $\kappa=(\sigma, \mu)$ such that $\mu_{ij} \leq n-1$ for all pairs $i,j \in \underline r$. This defines a filtration
\begin{equation*}
\mcK_1 \subset \dots \subset \mcK_n \subset \mcK_{n+1 } \subset \dots \subset \mcK=\colim_n \mcK_n
\end{equation*}
of $\mcK$ by suboperads.
\end{defn}

\begin{remark}
Considering $\mcK$-structures allows more control over the operads in question. A closer look at the $\mcK$-structure of $E$ yields that $\colim_{\kappa \in \mcK_n} E_{\kappa}  = E_n$: If $x=(\omega_0,\dots,\omega_l)\in \Sigma_{r}^{l+1}$ is in $E_n$, then $((\omega_0)_{ij},\dots,(\omega_l)_{ij})$ has at most $n-1$ variations for all $i,j\in \underline r$. Hence $x\in E_{\kappa}$ for $\kappa=(\omega_l,\mu)$ with $\mu_{ef}=n-1$ for all $i,j\in \underline r$. This will allow us to show that the differentials we are interested in restrict to $E_n$.
\end{remark}


\subsection*{Descending the twisted module to $E_n$-modules}

We assume that $n<\infty$ in this subsection.
In \cite{Fre11} the first author constructs a twisting differential $\partial_{\epsilon}$ on $\Sigma^{-n} T^n\circ E$. It is then shown that $\partial_{\epsilon}$ descends to $\Sigma^{-n} T^n \circ   E_n$  by using that $T^n$ can be interpreted as a $\mcK$-diagram and that hence $\Sigma^{-n}T^n \circ  E$ is a right $E$-module with $\mcK$-structure. We aim to prove a similar statement for the twisting differential $\partial_{\lambda}$. To achieve this, we will use the same $\mcK$-structures, hence we now recall the relevant definition.

For a complete graph $\kappa=(\sigma, \mu) \in \mcK(\underline e)$ and $\underline f \subset \underline e$ let
$\kappa\vert_{ \underline f} = (\sigma', \mu_{\underline f \times \underline f})$
be the complete subgraph of $\kappa$ with $\underline f\subset\underline e$
as vertex set
together with the ordering $\sigma' \colon \lbrace 1,\dots, \vert \underline f \vert \rbrace \ra \underline f$
defined by the composite
\begin{equation*}
\xymatrix{\lbrace 1,\dots, \vert \underline f \vert \rbrace \ar[r]^-{\cong} & \sigma^{-1}( \underline f) \ar[r]^-{\sigma} & \underline f.}
\end{equation*}

\begin{proposition}\cite[4.2]{Fre11}
There is a $\mcK$-diagram associated to $T^n$ defined as follows: For $n=1$ and $\kappa= (\sigma, \mu) \in \mcK(\underline{e})$ set $T^n_{\kappa} = k \cdot \sigma \subset T^1(\underline e)$.
For general $n$ and $\kappa=(\sigma, \mu) \in \mcK(\underline e)$ an element
\begin{equation*}
sx_1 \otimes\dots\otimes sx_l \in \Sigma T^{n-1}(\underline{e}_1) \otimes \dots \otimes \Sigma T^{n-1}(\underline{e}_l) \subset \overline{T}^c(\Sigma T^{n-1})(\underline e_1 \sqcup \dots \sqcup \underline e_l)
\end{equation*}
is in $T^n_{\kappa}$  if the following conditions hold:
\begin{enumerate}
\item For $1\leq i \leq l$ we have that $x_i$ is an element of $ T^{n-1}_{\kappa\vert_{\underline{e_i}}}$.
\item If $e,f \in \underline e$ with $\mu_{ef} < n-1$ then there exists $i$ such that $e,f \in \underline{e_i}.$
\item If $e,f\in \underline e$ with $\mu_{ef} = n-1$ and if $e \in \underline{e_i}, f \in \underline{e_j}$ with $i < j$ then $\sigma_{ef} = \id$.
\end{enumerate}
\end{proposition}

\begin{remark}
The idea behind this definition is the following (see \cite{FrApp}): Interpreting $t \in T^n(\underline{e})$ as a tree, the smallest complete graph $\kappa$ with $t \in T^n_{\kappa}$ has vertices ordered like the inputs of $t$ and weights $\mu_{ef}$ such that $n-1-\mu_{ef}$ equals the level on which the paths from $e$ and $f$ to the root first join.
\end{remark}

For a $\mcK$-operad $\mcP$ and  $\kappa$ a complete graph let $(T^n  \circ  \colim \mcP)_{\kappa}$ be generated as a $k$-module by
elements $t(p_1,\dots,p_l) \in  T^n \circ \colim \mcP$ such that $t \in T^n_{\kappa'}, p_i \in \mcP_{\kappa_i}$ with
$\kappa'(\kappa_1,\dots,\kappa_l) \leq \kappa$.
This makes $T^n \circ  \colim \mcP $  a right $ \colim \mcP$-module with $\mcK$-structure.

\begin{lemma}\label{prop:LambdaRespectsComKappa}
The twist $\partial_{\theta}$ satifies
\begin{equation*}
\partial_{\theta}( (T^n   \circ {\Com})_{\kappa} )\subset \bigoplus_{e\in \underline e}  U_{\Com}(\lbrace e\rbrace ) \otimes (T^n   \circ {\Com})_{\kappa\vert (\underline{e}\setminus \lbrace e \rbrace)}
\end{equation*}
for $\kappa = (\sigma, \mu) \in \mcK(\underline e)$.
\end{lemma}

\begin{proof}
We proceed by induction. For $n=1$ we see that $x \in T^1_{\kappa}$ if and only if  $x$ corresponds to the $1$-level tree $t_r(e_1,\dots,e_r)$ with $r$ leaves decorated by $e_1,\dots,e_r$, where $\sigma^{-1}(i) =e_i$. If $r>1$ the map $\theta$ sends $t_r(e_1,\dots,e_r)$ to
\begin{equation*}
- \mu_{\lbrace e_1\rbrace }^U \otimes t_{r-1}(e_2,\dots,e_r)  +(-1)^r \mu_{\lbrace e_r\rbrace }^U \otimes t_{r-1}(e_1,\dots,e_{r-1}).
\end{equation*}
Denote by $\partial_{\theta}^{(a)}$ the morphism $\partial_{\theta}$ defined on $T^a\circ {\Com}$. For $n>1$ observe that
\begin{equation*}
\partial_{\theta}^{(n)}(sx_1 \otimes\dots\otimes sx_r) =  \sum_j \pm sx_1 \otimes \dots \otimes s\partial_{\theta}^{(n-1)}(x_j) \otimes \dots \otimes sx_r
\end{equation*}
for $sx_1 \otimes \dots \otimes sx_r \in (\Sigma T^{n-1})^{\otimes r}$.
Let $x_i \in T^{n-1}_{\underline{f_i}}$ for all $i$. By the induction hypothesis
\begin{equation*}
\partial_{\theta} (x_i) \in \bigoplus_{e \in \underline{f}_i}  U_{\Com}(\lbrace e\rbrace ) \otimes (T^{n-1}  \circ {\Com})_{\kappa\vert \underline{f}_i\setminus \lbrace e \rbrace},
\end{equation*}
which yields the claim.
\end{proof}

We already noted that $T^n$ is $\Sigma_*$-free: There are graded modules $G(\underline e)$ with
$T^n(\underline e)\cong G(\underline e) \otimes \Sigma_{\underline e}$.
The functor $G$ is defined inductively by
\begin{equation*}
G^n=\bigoplus_{i\geq 1}(\Sigma G^{n-1})^{\otimes i}\quad\text{for $n>0$},
\quad\text{and}
\quad
G^0(\underline e) = \begin{cases}
k, &\vert \underline e \vert =1,\\
0, & \vert \underline e \vert \neq 1.
\end{cases}
\end{equation*}
Intuitively, $G$ associates to a finite set $\underline e$ the set of trees with $\vert \underline e\vert$ leaves with a tree $t\in G(r)$ having degree equal to the number of its edges. The inclusion $G^n(\underline e) \ra T^n(\underline e)$ is given by mapping a tree $t$ to the tree with leaves labeled by $\underline e$ according to the chosen order $\tau_{\underline e}$.

We now deduce from Proposition \ref{prop:LambdaRespectsComKappa} that $\partial_{\lambda}$ restricts to $U_{\Com} \otimes (\Sigma^{-n}T^n \circ E_n)$. Note that  the Lemmata \ref{lemma:Lambda0}, \ref{lemma:respectfulPartialLambda} and Proposition \ref{prop:LambdaRestrictsToEn} are completely analougous to \cite[5.2]{Fre04}, \cite[5.3]{Fre04} and \cite[5.4]{Fre04}.

For a complete graph $\kappa$ denote by $G^n_{\kappa}$ the arity-graded $k$-submodule of $G^n$ generated by elements $g \in G^n$ with $g \in T^n_{\kappa}$.

\begin{lemma}\label{lemma:Lambda0}
The map $\lambda_0$ satisfies
\begin{equation*}
\lambda_0 (G^n_{\kappa}) \subset \bigoplus_{\underline{e} = \underline{e}' \sqcup \underline{e}''}  U_{\Com}(\underline{e}' ) \otimes (T^n  \circ  E)_{\kappa\vert_{ \underline{e}''} }
\end{equation*}
for $\kappa =(\tau_{\underline e}, \mu) \in \mcK(\underline e)$ with $\underline e \subset \mathbb{N}_0$.
\end{lemma}

\begin{proof}
We know that $\lambda_0 = \tilde \iota (\theta + \eta_U\otimes \gamma)$ and that
\begin{equation*}
\theta(T^n_{\kappa}) \subset \bigoplus_{e \in \underline e}  U_{\Com}(\lbrace e\rbrace ) \otimes (T^n \circ  {\Com})_{\kappa\vert_{ \underline{e} \setminus \lbrace e \rbrace}},
\end{equation*}
while according to \cite[4.6]{Fre11}
\begin{equation*}
(\eta_U \otimes \gamma)(G^n_{\kappa}) \subset U_{\Com}(\emptyset) \otimes (T^n \circ {\Com})_{\kappa}.
\end{equation*}
But by \cite[4.5]{Fre11} $(T^n \circ {\Com})_{\kappa}$
is spanned by elements $t(c_1,\dots,c_l)$ with $t \in T^n_{\kappa'}$, $c_i \in \Com_{\kappa_i}$
such that $\kappa_i$ is also of the form $(\tau_{\underline{e}^{(i)}}, \mu^{(i)})$ for some  $\underline{e}^{(i)} \subset \mathbb{N}_0$. Hence we find that
\begin{equation*}
(T^n \circ \iota) ((T^n  \circ {\Com})_{\kappa}) \subset (T^n  \circ  {E})_{\kappa}.
\end{equation*}
Observe that $\kappa\vert_{\underline e\setminus \lbrace e \rbrace} = (\tau_{\underline e\setminus \lbrace e \rbrace}, \mu')$.
Hence
\begin{equation*}
(U_{\Com} \otimes (T^n \circ \iota))( \bigoplus_{e \in \underline e} U_{\Com}(\lbrace e\rbrace ) \otimes (T^n  \circ {\Com})_{\kappa\vert_{\underline e\setminus \lbrace e \rbrace}})
\subset  \bigoplus_{e\in \underline e}  U_{\Com}(\lbrace e\rbrace ) \otimes (T^n \circ  {E})_{\kappa\vert _{\underline e\setminus \lbrace e \rbrace}}
\end{equation*}
as well. This proves the claim.
\end{proof}

\begin{lemma}\label{lemma:respectfulPartialLambda}
The twist $\partial_{\lambda}$ satisfies
\begin{equation*}
\partial_{\lambda}(U_{\Com} \otimes(T^n\circ E)_{\kappa}) \subset \bigoplus_{ \underline{e}' \subset \underline{e}}  U_{\Com} \otimes (T^n  \circ E)_{\kappa\vert_{\underline{e}'}}
\end{equation*}
for all $\kappa \in \mcK(\underline e)$ with $\underline e \subset \mathbb{N}_0$.
\end{lemma}

\begin{proof}
We show by induction that
\begin{equation*}
\partial_{\lambda_m} (U_{\Com} \otimes (T^n \circ E)_{\kappa}) \subset \bigoplus_{\underline{e}' \subset \underline{e}}  U_{\Com}\otimes (T^n  \circ E)_{\kappa\vert_{\underline{e}'}}
\end{equation*}
for all $m$.
Let $t(e_1,\dots,e_l) \in (T^n \circ E)_{\kappa}$. Then there are $\kappa' \in \mcK(\underline f)$ and $\kappa_i \in \mcK(\underline f_i)$ such that $t\in T^n_{\kappa'}$ and $e_i \in E_{\kappa_i}$ with
$\kappa'(\kappa_1,\dots,\kappa_l) \leq \kappa$.

Note that $(T^n \circ E) (\underline e)= \bigoplus_{i\geq 1} T^n(i) \otimes_{\Sigma_i} (E^{\otimes i}(\underline e)),$
hence we can assume that $\kappa'=(\id_{\underline l}, \mu')$ and $t \in G^n_{\kappa'}$. Writing down the definition of $\partial_{\lambda_0}$ and using Lemma \ref{lemma:Lambda0} yields that
$\partial_{\lambda_0}$ maps $U_{\Com} \otimes (T^n \circ E)_{\kappa}$ to  $U_{\Com}\otimes \bigoplus_{\lbrace i_1<\dots<i_j \rbrace = \underline e' \subset \underline l} (T^n \circ E)_{\kappa' \vert_{\underline{e'}}(\kappa_{i_1},\dots,\kappa_{i_j})}$. Since
\begin{equation*}
\kappa' \vert_{\underline{e'}}(\kappa_{i_1},\dots,\kappa_{i_j}) \leq \kappa \vert_{\underline f_{i_1} \sqcup\dots\sqcup \underline f_{i_j}}
\end{equation*}
the claim holds for $m=0$.
For $m>0$ recall that
\begin{equation*}
\partial_{\lambda_m}= \sum_{a+b=m-1} \tilde \nu \partial_{\lambda_a} \partial_{\lambda_b}.
\end{equation*}
The induction hypothesis yields that
\begin{equation*}
\partial_{\lambda_a} \partial_{\lambda_b}(U_{\Com} \otimes (T^n \circ E)_{\kappa}) \subset \bigoplus_{\underline{e}' \subset \underline{e}}  U_{\Com} \otimes (T^n \circ  E)_{\kappa_{\vert \underline{e}'}}.
\end{equation*}
Since
\begin{equation*}
\tilde \nu (u \otimes t(e_1,\dots,e_l)) = \sum_i \pm u \otimes t(\iota \psi(e_1),\dots, \iota\psi(e_{i-1}), \nu (e_i), e_{i+1},\dots,e_l)
\end{equation*}
for $\xi \in T^n$ and $y_r \in E$ the same reasoning as in Lemma \ref{lemma:Lambda0} together with our assumptions on the interaction between $\psi,  \iota, \nu$ and the $\mcK$-structure yields the claim.
\end{proof}

\begin{proposition}\label{prop:LambdaRestrictsToEn}
We have
\begin{equation*}
\partial_{\lambda}(U_{\Com} \otimes  (T^n \circ {E_n}) ) \subset U_{\Com} \otimes (T^n \circ {E_n}).
\end{equation*}
\end{proposition}

\begin{proof}
We need to show that $\partial_{\lambda}(T^n) \subset U_{\Com} \otimes (T^n \circ E_n)$. Observe that for every $r \geq 0$ and $\xi \in T^n(r)$ there is a complete graph $\kappa=(\sigma,\mu)$ with $\xi \in T^n_{\kappa}$ such that $\mu_{ef}\leq n-1$ for all vertices $e,f$. Hence
\begin{equation*}
\partial_{\lambda}(T^n(r)) = \partial_{\lambda} (\colim_{\kappa \in \mcK_n(\underline{r})} T^n_{\kappa}) =  \colim_{\kappa \in \mcK_n(\underline r)}( \bigoplus_{\underline{e}' \subset \underline{r}}  U_{\Com} \otimes (T^n \circ  E)_{\kappa\vert_{\underline{e}'}}).
\end{equation*}
But
\begin{align*}
& \colim_{\kappa \in \mcK_n(\underline{r})}( \bigoplus_{\underline{e}' \subset \underline{r}}  U_{\Com}\otimes (T^n \circ  E)_{\kappa\vert_{\underline{e}'}})\\
= &  U_{\Com}  \otimes  \bigoplus_{\underline{e}' \subset \underline{r}}  \colim_{\kappa \in \mcK_n(\underline{r})} (T^n  \circ E)_{\kappa\vert_{\underline{e}'}}\\
\subset & U_{\Com} \otimes  \bigoplus_{\underline{e}' \subset \underline{r}}( T^n \circ E_n)(\underline{e}'),
\end{align*}
since $\kappa( \kappa_1,\dots,\kappa_l) \in \mcK_n$ implies $\kappa_1,\dots,\kappa_l \in \mcK_n$.
\end{proof}


\section{From modules over operads to Quillen homology}\label{section:ModuleToHomology}

\subsection*{The model category of left modules over an algebra in right modules over an operad}

Let $\mcP$ be an operad and $U$ an algebra in right $\mcP$-modules. We will define a model structure on $\UModP$ by applying the standard method
to transport cofibrantly generated model categories along adjunctions to the adjunction
\begin{equation*}
\xymatrix{F= U \otimes - \colon \MP \ar@<1ex>[r] & \ar@<1ex>[l]  \UModP \colon V}
\end{equation*}
with $V$ the corresponding forgetful functor. This model structure will allow us to compare $U_{\Com} \otimes \Sigma^{-n}B^n_{E_n}$ with a standard resolution for computing $E_n$-homology.

We start by determining the  $FI$- and $FJ$-cell complexes, where $FI$ (respectively $FJ$) is the set of maps
\begin{equation*}
U \otimes (i \otimes (F_r \circ \mcP)) \colon U \otimes (C \otimes (F_r \circ \mcP)) \ra U \otimes (D \otimes (F_r \circ \mcP))
\end{equation*}
with $i \colon C \ra D$ a generating cofibration (respectively a generating acyclic cofibration) in $\dgmod$ and
\begin{equation*}
F_r(l) = \begin{cases} k[\Sigma_r], & l=r, \\
0, & l \neq r.
\end{cases}
\end{equation*}
Recall that the generating cofibrations in $\dgmod$ are of the form $S^{d-1} \ra D^d$ and the acyclic cofibrations are of the form $0 \ra D^d$, with
$S^{d-1}_{d-1}=k$, $S^{d-1}_l =0$ otherwise and $D^d$ the acyclic chain complex with $D^d_{d-1}=D^d_d=k$ and $D^d_l$ otherwise.
Also note that the underlying differential graded module of the direct sum of left $U$-modules in right $\mcP$-modules is the direct sum of their underlying differential graded modules. Using this yields the following observations.

\begin{proposition}
An $FI$-cell attachment in $\UModP$ is an inclusion $K \ra (K \oplus G, \partial)$ with $G= U \otimes (M \circ \mcP)$ for a free $\Sigma_*$-module $M$ with trivial differential and $\partial \colon G \ra K$.
An $FJ$-cell attachment in $\UModP$ is an inclusion $K \ra K \oplus G'$ with  $G'= U \otimes (\bigoplus_{\alpha} D^{n_{\alpha}} \otimes (F_r \circ \mcP))$.
\end{proposition}

\begin{corollary}\label{cor:RelativeCellComplexes}
A relative $FI$-cell complex in $\UModP$ is an inclusion
\begin{equation*}
K \ra (K \oplus (U \otimes (M\circ \mcP)), \partial)
\end{equation*}
with $M$ a $\Sigma_*$-free $\Sigma_*$-module with trivial differential, such that $ K \oplus (U \otimes (M\circ \mcP)) $ is filtered by  $G_{\lambda}, \lambda < \kappa$ for a given ordinal $\kappa$, with $\partial( G_{\lambda}) \subset G_{\lambda-1}$ and $G_0 = K$.
A relative $FJ$-cell complex in $\UModP$ is the same as an $FJ$-cell attachment.
\end{corollary}

Now we are in the position to prove that the adjunction between $\UModP$ and the category of right $\mcP$-modules gives rise to a model structure on $\UModP$.

\begin{thm}
Let $\mcP$ be cofibrant in $\dgmod$. Let $\UModP$ be the category of left $U$-modules in right $\mcP$-modules.
Then $\UModP$ is a cofibrantly generated model category with weak equivalences and fibrations created  by $V \colon \UModP \ra \MP$. The generating (acyclic) cofibrations $FI$ and $FJ$ are of the form
\begin{equation*}
U \otimes (C \otimes (F_r\circ \mcP)) \ra U \otimes (D \otimes (F_r\circ \mcP))
\end{equation*}
with $C \ra D$ a generating (acyclic) cofibration in $\dgmod$.
\end{thm}

\begin{proof}
The category $\UModP$ is complete and cocomplete with limits and colimits created by $\UModP\ra \MP$. Let $f$ be a relative $FJ$-cell complex. Since $V$ creates colimits, we have that $f = F(g)$ for a relative $J$-cell complex $g$, which is an acyclic cofibration by \cite[11.1.8]{Fre09}. But the functor $F$ sends acyclic cofibrations to weak equivalences by \cite[Lemma 5.6]{BMR}.
The small object argument holds trivially for $FJ$ since the domains of $FJ$ are all $0$. The domains of $FI$ are of the form $U \otimes (D^l \otimes (F_r\circ \mcP))$, hence a morphism to $K \in \UModP$ is equivalent to picking an element $x \in K(r)$ of degree $l$. If $K = \colim_{\lambda < \kappa} L_{\lambda}$ is a relative $FI$-cell-complex, it is clear that $x \in L_{\lambda}$ for some $\lambda < \kappa$. But since $L_{\lambda} \in \UModP$ we see that  $U \otimes (D^l \otimes (F_r\circ \mcP))$ as a whole gets mapped to $L_{\lambda}$.
Hence by theorem \cite[11.3.2]{Hi03}
the category $\UModP$ is a model category with the properties stated above.
\end{proof}

By \cite[11.1.8]{Fre09} Corollary \ref{cor:RelativeCellComplexes} implies:

\begin{corollary}\label{cor:CofUModP}
Let $(U \otimes (M\circ \mcP), \partial)$ be a quasifree object in $\UModP$ such that $M$ is $\Sigma_*$-free and such that there is an ordinal $\kappa$ and a filtration $(G_{\lambda})_{\lambda < \kappa}$ of $U \otimes (M\circ \mcP)$ with $\partial (G_{\lambda}) \subset G_{\lambda-1}$. Then $(U \otimes (M\circ \mcP), \partial)$ is cofibrant. In particular such quasifree objects in $\UModP$ which are bounded below as chain complexes are cofibrant.
\end{corollary}

Finally we examine how an operad morphism $\mcQ \ra \mcP$ allows us to compare left modules in right $\mcQ$-modules and in right $\mcP$-modules.

\begin{proposition}
Given a morphism $\mcQ \ra \mcP$ of operads, let $(V,\mu_V, \eta_V)$ be an algebra in right $\mcQ$-modules and $(N,\mu_N, \gamma_N) \in \VModQ$. Then $V \circ_{\mcQ} \mcP$ is an algebra in right $\mcP$-modules with multiplication
\begin{equation*}
\xymatrix{ (V\circ_{\mcQ} \mcP) \otimes (V\circ_{\mcQ} \mcP) \ar[r]^-{\cong} & (V \otimes V) \circ_{\mcQ} \mcP \ar[r]^-{\mu_V\circ_{\mcQ} \mcP} &V\circ_{\mcQ} \mcP }
\end{equation*}
and unit defined via the inclusion $V \ra V \circ_{\mcQ} \mcP$.
Furthermore, $N \circ_{\mcQ} \mcP$ is a left $V \circ_{\mcQ} \mcP$-module in right $\mcP$-modules with structure maps
\begin{equation*}
\xymatrix{
(V \circ_{\mcQ} \mcP) \otimes (N \circ_{\mcQ} \mcP) \ar[r]^-{\cong} & (V \otimes N) \circ_{\mcQ} \mcP \ar[r]^-{\mu_N\circ_{\mcQ} \mcP} & N \circ_{\mcQ} \mcP }
\end{equation*}
and
\begin{equation*}
\xymatrix{ (N \circ_{\mcQ} \mcP) \mcP \cong N \circ_{\mcQ} (\mcP \mcP) \ar[r]^-{N \circ_{\mcQ} \gamma_{\mcP}} & N \circ_{\mcQ} \mcP. }
\end{equation*}
\end{proposition}

For categories of right modules a morphism of operads gives rise to a Quillen adjunction, see \cite[Theorem 16.B]{Fre09}. In our setting we have a similar result:

\begin{proposition}\label{prop:QuiEqUModP}
Let $V$ be an algebra in right $\mcQ$-modules. A morphism $\mcQ \ra \mcP$ of operads gives rise to an adjunction
\begin{equation*}
\xymatrix{ - \circ_{\mcQ} \mcP \colon \VModQ \ar@<1ex>[r] & {}_{V \circ_{\mcQ} \mcP} \mathrm{Mod}(\mathcal{M}_{\mcP}) \ar@<1ex>[l] \colon \Res, }
\end{equation*}
where for $M \in {}_{V \circ_{\mcQ} \mcP} \mathrm{Mod}(\mathcal{M}_{\mcP})$ the structure maps of $\Res(M)$ are defined by restricting the right $\mcP$-module structure to $\mcQ$ and via the map $V \ra V\circ_{\mcQ} \mcP$.
If $\mcP$ and $\mcQ$ are cofibrant as differential graded modules in each arity, this is a Quillen adjunction.
\end{proposition}


\subsection*{The twisted module with coefficients associated to the operadic bar construction}

We now use the operadic bar construction to obtain an object in $\UCModMEn$ which can be used to calculate $H_*^{E_n}(A;A_+)$. We then show that this object admits a trivial fibration to $\Omega^1_{\Com}$, which will allow us to compare it with $(U_{\Com} \otimes (\Sigma^{-n} T^n \circ {E_n}), \partial_{\lambda})$.

We first review the right $\mcP$-modules modeling the universal enveloping algebra and the module of K\"ahler differentials.

\begin{proposition}
(see \cite[10.2]{Fre09})\label{prop:UnivEnv}
Set $U_{\mcP}(i) = \mcP(i+1)$. This is a right $\mcP$-module with structure map given by
\begin{equation*}
\gamma_{U_{\mcP}}(p; p_1,\dots,p_i) = \gamma_{\mcP}(p; 1, p_1,\dots,p_i)
\end{equation*}
for $p \in U_{\mcP}(i), p_l \in \mcP$.
Furthermore, $U_{\mcP}$ is an associative algebra in right $\mcP$-modules with multiplication induced by the partial composition $\circ_1$.
We will write
\begin{equation*}
p(p_1,\dots,p_{i-1},x,p_i,\dots,p_{l-1})
\end{equation*}
for the element $\gamma_{U_{\mcP}}(p\cdot(1\dotsi); p_1,\dots,p_{l-1}) \in \mcP[1]$ with $(1\dotsi) \in \Sigma_i$ the cyclic permutation.
Note that for any $\mcP$-algebra $A$ we have that $U_{\mcP}(A) = U_{\mcP} \circ_{\mcP} A$ with the induced algebra structure.
\end{proposition}

For the following proposition we assume that $\mcP(i)$ is a $k$-module to avoid additional signs.

\begin{proposition}\label{lem:KaehlerRightModule}(see \cite[10.3]{Fre09})
Let $\mcP(i)$ be concentrated in degree zero for each $i\geq 0$.
There is a right $\mcP$-module $\Omega^1_{\mcP}$ such that  $\Omega^1_{\mcP}(A) = \Omega^1_{\mcP} \circ_{\mcP} A$ for all $\mcP$-algebras $A$.
As a $k$-module $\Omega^1_{\mcP}$ is generated by expressions
\begin{equation*}
p(x_{i_1},\dots,dx_{i_j},\dots,x_{i_l})
\end{equation*}
with $p \in \mcP(l)$, $\lbrace i_1,\dots,i_l \rbrace = \lbrace 1,\dots,l \rbrace$ and indeterminates $x_1,\dots,x_l$. These have to fulfill equivariance relations generated by
\begin{equation*}
(p\cdot\sigma) (x_1,\dots,dx_i, \dots,x_l) -  p(\sigma. (x_1,\dots,dx_i,\dots,x_l))
\end{equation*}
for all $\sigma \in \Sigma_l$ with
$\sigma\cdot(y_1,\dots,y_l) = (y_{\sigma^{-1}(1)},\dots,y_{\sigma^{-1}(l)})$. The right $\Sigma$-action is defined by
\begin{equation*}
(p(x_1,\dots,dx_i,\dots,x_l))\cdot\sigma = p(x_{\sigma(1)},\dots,dx_{\sigma(i)},\dots,x_{\sigma(l)})
\end{equation*}
and the right $\mcP$-module structure is determined by
\begin{equation*}
p(x_1,\dots,dx_i,\dots,x_l) \circ_j q =
\begin{cases}
(p \circ_j q)(x_1,\dots,dx_{i+m-1},\dots,x_{l+m-1}), & j<i,\\
(p \circ_j q)(x_1,\dots,dx_i,\dots,x_{l+m-1}), & j>i,\\
\sum_{r=0}^{m-1} (p \circ_j q)(x_1,\dots,dx_{r+i},\dots,x_{l+m-1}), & j=i.
\end{cases}
\end{equation*}
for $p \in \mcP(l), q \in \mcP(m)$.
The $U_{\mcP}$-module structure is given by
\begin{equation*}
q \cdot p(x_1,\dots,dx_i,\dots,x_l) = (p \circ_1 q)(x_1,\dots,dx_i,\dots,x_{l+m})
\end{equation*}
for $q \in U_{\mcP}(m) =\mcP(m+1)$.
\end{proposition}

\begin{remark}
For an operad $\mcP$, the right $\mcP$-module $U_{\mcP}$ modeling universal enveloping algebras is an algebra in right $\mcP$-modules and the K\"ahler differentials $\Omega^1_{\mcP}$ form a left $U_{\mcP}$-module in right $\mcP$-modules.
\end{remark}

\begin{example}
Applied to $\mcP = \Com$ we have that
\begin{equation*}
\Omega^1_{\Com}(\underline e)= k < \mu_{\underline e}(dx_{e},x,\dots,x) \vert e \in \underline e>.
\end{equation*}
The morphism induced by a bijection $\phi \colon \underline e \ra \underline f$ maps $\mu_{\underline e}(dx_e,x,\dots,x)$ to $ \mu_{\underline f}(dx_{\phi(e)},x,\dots,x)$.
The right $\Com$-module structure is given by
\begin{equation*}
\mu_{\underline e}(dx_e,x,\dots,x) \circ_g \mu_{\underline f} = \begin{cases}
\mu_{(\underline e \sqcup \underline f)\setminus \lbrace g \rbrace}(dx_e,\dots,dx,\dots,x), & g\neq e,\\
\sum_{f \in \underline f} \mu_{(\underline e \sqcup \underline f)\setminus \lbrace e \rbrace}(dx_f,x,\dots,x), & g=e
\end{cases}
\end{equation*}
for $e,g\in \underline e$.
The algebra $U_{\Com}$ acts on $\Omega^1_{\Com}$ by
\begin{equation*}
\mu_{\underline e}^U \cdot \mu_{\underline f}(dx_f,x,\dots,x)  = \mu_{\underline e \sqcup \underline f}(dx_f,x,\dots,x).
\end{equation*}
\end{example}

We now construct a cofibrant replacement of $\Omega^1_{\Com}$ via the operadic bar construction, which we will later compare with the twisted module associated to the iterated bar complex.

\begin{defn}(\cite[3.1.9]{Fre04})
Let $(\mcP, \gamma_{\mcP})$ be an operad with $\mcP(0) =0$, $\mcP(1) =k$. Let $\bar \mcP$ be the augmentation ideal of $\mcP$. The reduced bar construction $\bB(\mcP)$ is the quasifree cooperad
\begin{equation*}
\bB(\mcP)= (\mcF^c(\Sigma \bar \mcP), \partial_B)
\end{equation*}
with $\partial_B \colon \mcF^c(\Sigma \bar \mcP) \ra \mcF^c(\Sigma \bar \mcP)$
the coderivation of cooperads which corresponds to the map $\mcF^c(\Sigma \bar \mcP) \ra  \bar \mcP$
 of degree $-1$
given by
\begin{equation*}
\xymatrix{ \mcF^c(\Sigma \bar \mcP) \ar@{>>}[r] & \mcF^c_{(2)}(\Sigma \bar \mcP)\ar[r]^-{\cong} & \Sigma^2 \bar \mcP(I; \bar \mcP) \ar[r]^-{\Sigma^2 \gamma_{\mcP}} & \Sigma^2 \bar \mcP \ar[r]^-{\cong} &  \Sigma \bar \mcP.
}\end{equation*}
Here $\mcF^c_{(2)}(\Sigma \bar \mcP)$ denotes the summand $\Sigma \bar \mcP (I; \Sigma \bar \mcP)$ of weight $2$ in the decomposition $\mcF^c(\Sigma \bar \mcP)= \bigoplus_{i\geq 0} \mcF^c_{(i)}(\Sigma \bar \mcP)$ of the cofree cooperad.
\end{defn}

Set ${\bB}_{(i)}(\mcP) = \mcF^c_{(i)}(\Sigma \bar \mcP)$.  Note that this weight grading is not respected by the differential of $\bB(\mcP)$.

\begin{defn}(\cite[4.4]{Fre04})
The differential graded $\mcP$-bimodule $B(\mcP,\mcP,\mcP)$
is given by
\begin{equation*}
B(\mcP, \mcP, \mcP) = (\mcP \circ  \bB(\mcP)  \circ \mcP,   \partial_L  + \partial_R  ),
\end{equation*}
with the left and right $\mcP$-module derivation  $\partial_L \colon \mcP \circ \bB(\mcP)\circ   \mcP \ra \mcP\circ \bB(\mcP)  \circ \mcP$
induced by the map
\begin{equation*}
\xymatrix{ \bB_{(i)}(\mcP) \ar@{>>}[r] & \bB_{(i-1)}(\mcP)(I; \Sigma \bar \mcP) \ar[r]^-{\cong} &  \bB_{(i-1)}(\mcP)(I; \bar \mcP)  \ar@{^{(}->}[r] & \bB(\mcP)  \circ \mcP. }
\end{equation*}
Here the first map sends an element $x \in \bB(\mcP)$ of the form $ x= (b;s p_1,\dots,s p_r)$ with $s p_i \in \Sigma\bar \mcP$ and $b\in \bB(\mcP)$ to
\begin{equation*}
\sum_{j=1}^r \pm ((b;s p_1,\dots,s p_{j-1},1,s p_{j+1},\dots,s p_r);1,\dots,1,p_j,1,\dots,1).
\end{equation*}
The left and right $\mcP$-module derivation $\partial_R \colon \mcP \circ  \bB(\mcP) \circ \mcP  \ra \mcP  \circ \bB(\mcP)\circ \mcP$
is induced by the map
\begin{equation*}
\xymatrix{ \bB(\mcP) \ar[r] & \mcP  \circ \bB(\mcP) }
\end{equation*}
which maps $(s p;b_1,\dots,b_s) \in \bB(\mcP)$ with $s p \in \Sigma \bar \mcP$ and $b_i \in \bB(\mcP)$ to $(p;b_1,\dots,b_s)$.
For the exact signs see \cite[4.4.3]{Fre04}.
\end{defn}

\begin{defn}(\cite[4.4]{Fre04})
Let $\mcP$ be an operad, $L$ a left $\mcP$-module and $R$ a right $\mcP$-module.
The differential graded $\Sigma_*$-module $B(R,\mcP,L)$
is given by
\begin{equation*}
B(R, \mcP, L) = R \circ_{\mcP} B(\mcP, \mcP, \mcP) \circ_{\mcP} L.
\end{equation*}
We denote by $\epsilon_B$ the augmentation
\begin{equation*}
\epsilon_B \colon B(R, \mcP, L) \ra R \circ_{\mcP} L.
\end{equation*}
\end{defn}

The object $B(R,\mcP,L)$ inherits a grading by weight components $B_{(i)}(R,\mcP, L)$ from $\bB(\mcP)$. The summand $B_{(i)}(R,\mcP, L)$ corresponds to expressions in $R  \circ \mcF^c(\Sigma \bar \mcP) \circ  L$ with $i$ occurences of elements in $\bar \mcP$.

From the identity
\begin{equation*}
\Omega_{\mcP}(B) = (U_{\mcP}(B)\otimes Y, \partial'_{\alpha})
\end{equation*}
for a quasifree $\mcP$-algebra $B=(\mcP(Y),\partial_{\alpha})$ (see \eg \cite[2.1.1]{Ho10}) and since $U_{\Com} \circ_{E_n} A = U_{\Com}(A)$ for a commutative algebra $A$ we deduce:

\begin{lemma}\label{lem:HomComplexSimple}
Let $1\leq n \leq \infty$. If $n=\infty$ let $\epsilon_{{\Com}}$ denote the map  $E \ra {\Com}$, otherwise let $\epsilon_{{\Com}}$ denote the composite  $E_n \ra E \ra {\Com}$. Given $b \in B(I,E_n,I)$ considered as an element of $B(E_n,E_n,E_n)$, assume that $\partial_R(b)$ has an expansion
\begin{equation*}
\partial_R(b) = \sum_i e^{(i)}(b^{(i)}_1,\dots,b^{(i)}_{k_i}) \in B(E_n,E_n,I)
\end{equation*}
with $e^{(i)} \in E_n$, $b^{(i)}_1,\dots,b^{(i)}_{k_i} \in B(I,E_n,I)$. Consider the element
\begin{equation*}
\sum_i e^{(i)}(b^{(i)}_1,\dots,b^{(i)}_{j-1},x,b^{(i)}_{j+1},\dots,b^{(i)}_{k_i})
\end{equation*}
in $U_{E_n}(B(I,E_n,I))$.
We define
\begin{equation*}
\partial_{\theta_B} \colon U_{\Com} \otimes B(I, E_n,E_n) \ra U_{\Com} \otimes B(I,E_n,E_n)
\end{equation*}
to be the morphism in $\UCModMEn$ induced by
\begin{equation*}
\theta_B( b) = \sum_{i,j} \epsilon_{{\Com}}(e^{(i)})(\epsilon_{ {\Com}}\epsilon_Bb^{(i)}_1,\dots,\epsilon_{ {\Com}}\epsilon_Bb^{(i)}_{j-1}, x , \epsilon_{ {\Com}}\epsilon_Bb^{(i)}_{j+1},\dots,\epsilon_{{\Com}}\epsilon_Bb^{(i)}_{k_i}) \otimes b^{(i)}_j.
\end{equation*}
Then for $A$ a commutative algebra in $\dgmod$ we find that
\begin{equation*}
(U_{\Com} \otimes B(I, E_n,E_n), \theta_B) \circ_{E_n} A \cong U_{\Com}( A) \otimes_{U_{E_n}(Q_A)} \Omega^1_{E_n}(Q_A)
\end{equation*}
with $Q_A = B(E_n, E_n,A)$.
\end{lemma}

Since $(U_{\Com}\otimes B(I,E_n,E_n), \partial_{\theta_B})$ is quasifree in $\UCModMEn$ we know from Corollary \ref{cor:CofUModP} that $(U_{\Com} \otimes B(I,E_n,E_n), \partial_{\theta_B})$ is cofibrant in  $\UCModME$.

\begin{proposition}
We define a morphism
\begin{equation*}
\ev \colon (U_{\Com} \otimes B(I,E_n,E_n), \partial_{\theta_B}) \ra \Omega^1_{\Com}
\end{equation*}
of left $U_{\Com}$-modules in right $E_n$-modules as follows:  Restricted to $B(I,E_n,I)$ the map $\ev$ is
\begin{equation*}
\xymatrix{B(I,E_n,I) \ar@{->>}[r] & B_{(0)} (I,E_n,I) =I \ar[r] & \Omega^1_{\Com}}
\end{equation*}
where the last map sends $1\in I(\lbrace e \rbrace)$ to $\mu_{\lbrace e \rbrace}(dx_e)$. This yields a well defined morphism in $\UCModMEn$.
\end{proposition}

\begin{proof}
By definition $\ev$ maps $B_{(i)}(I,E_n,I)$ to zero for $i\geq 1$, hence since $B(I,E_n,E_n)$ is a quasifree right $E_n$-module it suffices to show that
\begin{equation*}
ev (\partial_B + \partial_L + \theta_B) = 0
\end{equation*}
on $B_{(1)}(I,E_n,I) \subset B(I,E_n,E_n).$  Let $a \in \bar E_n (\underline e)= (I  \circ \bar E_n \circ  I) (\underline e)= B_{(1)}(I,E_n,I)(\underline e)$. Note that $\partial_B$ vanishes on $B_{(1)}(I,E_n,E_n)$.  Both $\ev \partial_L$ as well as $\ev \theta_B$ map $a$ to
\begin{equation*}
\pm \sum_{e\in \underline e}  \epsilon_{ {\Com}}(a) (dx_e,x,\dots,x),
\end{equation*}
with opposite signs.
\end{proof}

\begin{lemma}(\cite[10.3]{Fre09})\label{ref:UCToOmega}
There is an isomorphism
\begin{equation*}
U_{\mcP} \otimes I \ra \Omega^1_{\mcP}
\end{equation*}
of left $U_{\mcP}$-modules given by mapping $1\in I(\lbrace e \rbrace)$ to $\mu_{\lbrace e \rbrace}(dx_e)$.
\end{lemma}

\begin{proposition}
The morphism
\begin{equation*}
\ev\colon(U_{\Com} \otimes B(I,E_n,E_n), \partial_{\theta_B})\ra \Omega^1_{\Com}
\end{equation*}
is a weak equivalence.
\end{proposition}
\begin{proof}
Filter $(U_{\Com} \otimes B(I,E_n,E_n),\partial_{\theta_B})$ by
\begin{equation*}
F^p = \bigoplus_{i\geq p} U_{\Com}(i)\otimes B(I,E_n,E_n)
\end{equation*}
and $\Omega^1_{\Com}$ by
\begin{equation*}
G^p= \bigoplus_{i\geq p} \Im (U_{\Com}(i) \otimes I)
\end{equation*}
where $\Im (U_{\Com}(i) \otimes I)$ is the image of $U_{\Com}(i) \otimes I$ under the isomorphism $U_{\Com} \otimes I \ra \Omega^1_{\Com}$ defined in Lemma \ref{ref:UCToOmega}. The morphism $\ev$ respects this filtration. We consider the associated spectral sequences.
Observe that the only part of the differential of $(U_{\Com} \otimes B(I,E_n,E_n),\partial_{\theta_B})$ that maps $F^p$ to $F^{p+1}$ is the part induced by $\theta_B$. Hence the $E^1$-term of the spectral sequence associated to the filtration $F$ is given by
\begin{equation*}
E^1_{p,q} = U_{\Com}(p) \otimes H_{q}(B(I,E_n,E_n)).
\end{equation*}
But the map $E^1(\ev)$ coincides with the tensor product of the identity and the augmentation $B(I,E_n,E_n) \ra I$ composed with the isomorphism defined in Lemma \ref{ref:UCToOmega} . According to \cite[4.1.3]{Fre04} this is a quasiisomorphism.
\end{proof}


\subsection*{The proof of the comparison results}

We defined a twisting morphism $\partial_{\theta}$ on $U_{\Com} \otimes \Sigma^{-n}B^n_{\Com}$ in Definition \ref{defn:theta} and Proposition \ref{defn:thetaForInfty} and showed in Proposition \ref{prop:LambdaRestrictsToEn} that the sum of $\partial_{\theta}$ and the differential $\partial_{\gamma}$ of $\Sigma^{-n}B^n_{\Com}$ can be lifted to a differential
\begin{equation*}\partial_{\lambda} \colon U_{\Com} \otimes (\Sigma^{-n} T^n\circ{E_n} ) \ra U_{\Com} \otimes (\Sigma^{-n} T^n\circ{E_n}).\end{equation*}
 We will construct a trivial fibration
\begin{equation*}(U_{\Com} \otimes (\Sigma^{-n} T^n\circ{E_n}), \partial_{\lambda}) \ra \Omega^1_{\Com}\end{equation*}
which will then allow us to compare $(U_{\Com} \otimes (\Sigma^{-n} T^n\circ{E_n}), \partial_{\lambda})$ and $(U_{\Com} \otimes B(I,E_n,E_n), \partial_{\theta_B})$ and to deduce that $(M \otimes_{A_+} A_+ \otimes \Sigma^{-n} B^n(A), \partial_{\theta})$ computes $E_n$-homology of $A$ with coefficients in $M$.

\begin{defn}
For $1 \leq n \leq \infty$ we define a morphism
\begin{equation*}
\Phi \colon (U_{\Com} \otimes (\Sigma^{-n} T^n\circ {E_n}), \partial_{\lambda}) \ra \Omega^1_{\Com}
\end{equation*}
in $\UCModME$  as follows:  Restricted to  the generators $\Sigma^{-n}T^n$  let $\Phi$ be the map
\begin{equation*}
\xymatrix{\Sigma^{-n}T^n \ar[r] & I \ar[r] &  \Omega^1_{\Com}}
\end{equation*}
with the first map given by mapping the element represented by the trunk tree $[0] \ra \dots \ra [0]$ labeled by $e$ to $1\in I(\lbrace e\rbrace)$ and the second map sending $1\in I(\lbrace e\rbrace)$ to $\mu_{\lbrace e\rbrace}(dx_e)$.
\end{defn}

\begin{lemma}
The map $\Phi$ is a chain map.
\end{lemma}
\begin{proof}
Observe that $\Phi$ is zero on $\Sigma^{-n}T^n(\underline e)$ unless $\underline e$ is a singleton. Also note that $\Phi$ factors as
\begin{equation*}
(U_{\Com} \otimes (\Sigma^{-n}T^n\circ E_n), \partial_{\lambda})
\xrightarrow{U \otimes (\Sigma^{-n} T^n \circ \epsilon_{{\Com}})}
(U_{\Com} \otimes (\Sigma^{-n} T^n\circ {\Com}), \partial_{\theta} + U_{\Com} \otimes \partial_{\gamma})
\xrightarrow{\Phi'}
\Omega^1_{\Com}
\end{equation*}
with $\Phi'$ the morphism in $\UCModMC$ induced by $\Phi_{\vert \Sigma^{-n}T^n}$.
A closer look at the maps $\theta$ and $\gamma$ reveals that we only have to prove that
we have the identity
$\Phi(\theta + \eta_U \otimes \gamma) =0$
on the object $\Sigma^{-n}T^n(\underline e)$
and for a labeling set with 2 elements $\underline e = \lbrace e_1,e_2 \rbrace$.
Consider the decorated tree $t(e_1,e_2)\in T^n(\lbrace e_1,e_2 \rbrace)$ with $t = [1] \ra [0] \ra\dots\ra [0]$ decorated by $e_1$ and $e_2$. Then $\Phi  \theta$ and $\Phi (\eta_U \otimes \gamma)$ both map $t(e_1,e_2)$ to
$\pm(\mu_{\lbrace e_1,e_2\rbrace}(dx_{e_1},x) + \mu_{\lbrace e_1,e_2\rbrace}(dx_{e_2},x))$
with opposite signs.
\end{proof}

\begin{proposition}The morphism
$\Phi \colon (U_{\Com} \otimes (\Sigma^{-n} T^n\circ{E_n}), \partial_{\lambda}) \ra \Omega^1_{\Com}$
is a weak equivalence.
\end{proposition}
\begin{proof}
Recall from the proof of Proposition \ref{prop:lifting}
that we have $\lambda = \sum_{m\geq 0} \lambda_m$
with
$\lambda_0= \tilde{\iota}(\theta + \eta_{U_{\Com}} \otimes \lambda)$
and
$\lambda_m = \sum_{a+b=m-1} \tilde \nu \partial_{\alpha_a} \alpha_b$.
Hence $\partial_{\lambda} =U_{\Com} \otimes \partial_{\epsilon} + \partial'$ with $\partial_{\epsilon}$ the differential of $B^n_E$ and
such that $\partial'$ lowers the arity of $U_{\Com}$.
Filter $(U_{\Com} \otimes (\Sigma^{-n} T^n\circ E_n), \partial_{\lambda} )$ by the subcomplexes
\begin{equation*}
F^p= \bigoplus_{i \geq p} U_{\Com}(i) \otimes (\Sigma^{-n} T^n \circ E_n)
\end{equation*}
and filter $\Omega^1_{\Com}$ by
\begin{equation*}
G^p= \bigoplus_{i\geq p} \Im (U_{\Com}(i) \otimes I).
\end{equation*}
Here $\Im (U_{\Com}(i) \otimes I)$ again is the image of $U_{\Com}(i) \otimes I$ under the isomorphism given in Lemma \ref{ref:UCToOmega}. The morphism $\Phi$ respects these filtrations.
The spectral sequence associated to the filtration $F$ has $E^1$-term
\begin{equation*}
E_{p,q}^1= U_{\Com}(p) \otimes H_q(\Sigma^{-n} B^n_{E_n}).
\end{equation*}
Let $\epsilon$ denote the quasiisomorphism
\begin{equation*}
\Sigma^{-n}B^n_{E_n} \ra I
\end{equation*}
exhibited in \cite[8.1]{Fre11} for $n < \infty$ and in \cite[ch.9]{Fre11} for $n=\infty$.
The map $\Phi$ factors as the map $U_{\Com} \otimes \epsilon$ followed by the isomorphism from Lemma \ref{ref:UCToOmega}, hence induces an isomorphism at the $E^1$-stage of the spectral sequences.
\end{proof}

\begin{thm}\label{th:EnHomWithUnCoef}
Let $1 \leq n \leq \infty$ and let $A$ be a commutative algebra in $\dgmod$ such that $A$ is cofibrant as a differential graded $k$-module.
Then we have:
\begin{equation*}
H_*^{E_n}(A; U_{\Com}(A)) = H_*(A_+ \otimes \Sigma^{-n}B^n(A), \partial_{\theta}).
\end{equation*}
\end{thm}
\begin{proof}
By definition $H_*^{E_n}(A;U_{\Com}(A))$ is the homology of
\begin{equation*}
U_{\Com}(A) \otimes_{U_{E_n}(Q_A)} \Omega^1_{E_n}(Q_A)
\end{equation*}
for a cofibrant replacement $Q_A$ of $A$ as an $E_n$-algebra. Since $A$ is a cofibrant in $\dgmod$, we can set $Q_A = B(E_n,E_n,A)$ (see \cite[2.12]{Fre09a}).
There exists a lift $f$ such that
\begin{equation*}
\xymatrix{ &(U_{\Com} \otimes  B(I,E_n,E_n), \partial_{\theta_B}) \ar@{>>}[d]_-{\sim}^-{\ev}\\
(U_{\Com} \otimes (\Sigma^{-n}T^n\circ{E_n}), \partial_{\lambda})  \ar@{>>}[r]^-{\Phi}_-{\sim} \ar[ru]^-{f}&  \Omega^1_{\Com} }
\end{equation*}
commutes since $(U_{\Com} \otimes (\Sigma^{-n}T^n\circ{E_n}),\partial_{\lambda}) $ is cofibrant according to Proposition \ref{cor:CofUModP}. The map $f$ is a quasiisomorphism of left $U_{\Com}$-modules in right $E_n$-modules. But note that while these are cofibrant objects in $\UCModMEn$ they are not cofibrant as right $E_n$-modules, hence we can not deduce from \cite[15.1.A]{Fre09} that $f \circ_{E_n} B$ is a quasiisomorphism for any $E_n$-algebra $B$.

However, consider $f \circ_{E_n} \Com$ and note that $- \circ_{E_n} \Com$ is the left adjoint in the Quillen adjunction
discussed in \ref{prop:QuiEqUModP}.
Hence $-\circ_{E_n}\Com$ preserves cofibrant objects and weak equivalences between them, and therefore $f \circ_{E_n} \Com$ is a weak equivalence. Now
\begin{equation*}
(U_{\Com} \otimes  B(I,E_n,E_n), \partial_{\theta_B}) \circ_{E_n} \Com \cong (U_{\Com} \otimes  B(I,E_n,\Com), \partial_{\theta_B})
\end{equation*}
and
\begin{equation*}
(U_{\Com} \otimes (\Sigma^{-n}T^n\circ{E_n}),\partial_{\lambda}) \circ_{E_n} \Com \cong (U_{\Com} \otimes (\Sigma^{-n}T^n\circ{\Com}), U_{\Com} \otimes \partial_{\gamma} + \partial_{\theta})
\end{equation*}
are quasi-free right $\Com$-modules because $U_{\Com}$ is a free right $\Com$-module generated by $\mu_1^U$ in arity zero and $\mu_2^U$ in arity one.
Therefore, according to Theorem \cite[15.1.A.(a)]{Fre09}, for a commutative algebra $A$ the map $f \circ_{E_n} \Com \circ_{\Com} A$ is a quasiisomorphism as well.
But for commutative $A$ we have $f \circ_{E_n} \Com \circ_{\Com} A = f\circ_{E_n} A$ and $U_{\Com} \circ_{E_n} A = U_{\Com}(A) =A_+$.
Since
\begin{equation*}
(U_{\Com} \otimes(\Sigma^{-n}T^n\circ{E_n}),\partial_{\lambda}) \circ_{E_n} A = (A_+\otimes \Sigma^{-n}T^n (A),\partial_{\theta} + \id_{A_+} \otimes \partial_{\gamma}),
\end{equation*}
while we know from Lemma \ref{lem:HomComplexSimple} that
\begin{equation*}
(U_{\Com} \otimes B(I,E_n,E_n),\partial_{\theta_B}) \circ_{E_n}A = U_{\Com}( A) \otimes_{U_{E_n}(Q_A)} \Omega^1_{E_n}(Q_A),
\end{equation*}
this yields an isomorphism
\begin{equation*}
H_*^{E_n}(A; U_{\Com}(A)) \cong H_*(A_+ \otimes \Sigma^{-n}T^n(A), \partial_{\theta} + \id_{A_+} \otimes \partial_{\gamma}).
\end{equation*}
\end{proof}

\begin{thm}[{Theorem~\ref{mainresult}, homological case}]\label{th:EnHomWithCoeffs}
Let $1 \leq n \leq \infty$. Let $A$ be a commutative differential graded algebra and let $M$ be a symmetric $A$-bimodule in $\dgmod$.
We have
\begin{equation*}
H_*^{E_n}(A; M) = H_*(B^{[n]}_*(A,M))
\end{equation*}
provided that $A$ is cofibrant in $\dgmod$.
\end{thm}
\begin{proof}
For $Q_A$ a cofibrant replacement of $A$ as an $E_n$-algebra $H_*^{E_n}(A; M)$ is the homology of the complex
\begin{equation*}
M \otimes_{A_+} A_+ \otimes_{U_{\Com}(Q_A)} \Omega^1_{E_n}(Q_A).
\end{equation*}
Again, we set $Q_A= B(E_n,E_n,A)$ and see that this equals
\begin{equation*}
M \otimes_{A_+} (A_+ \otimes B(I,E_n,A), \partial_{\theta_B}).
\end{equation*}
Since both $A_+ \otimes B(I,E_n,A)$ as well as $(A_+ \otimes \Sigma^{-n}B^n(A),\partial_{\theta})$ are cofibrant differential graded $A_+$-modules,
 the result follows directly from the quasiisomorphism exhibited in the proof of Theorem \ref{th:EnHomWithUnCoef}.
\end{proof}

\begin{thm}[{Theorem~\ref{mainresult}, cohomological case}]\label{th:EnCohWithCoeffs}
Let $1 \leq n \leq \infty$. Let $A$ be a commutative algebra and let $M$ be a symmetric $A$-bimodule.
We have:
\begin{equation*}
H^*_{E_n}(A; M) = H_*(B_{[n]}^*(A,M))
\end{equation*}
provided that $A$ is cofibrant in $\dgmod$.
\end{thm}
\begin{proof}
By definition $H^*_{E_n}(A;M) = H_*(\Der_{\mcP}(Q_A, M))$ for a cofibrant replacement $Q_A$ of $A$ as an $E_n$-algebra. Choose $Q_A=B(E_n,E_n,A)$.
Since $B(E_n,E_n,A)$ is quasifree,
\begin{equation*}
\Der_{E_n}(Q_A, M) = (\underline{\Hom}_k(B(I,E_n,A),M), \partial)
\end{equation*}
with $\partial(f)$ the composite
\begin{multline*}
B(I,E_n,A)\xrightarrow{\partial_R} E_n(B(I,E_n,A))
\xrightarrow{E_n \circ' B(I,E_n,A)} E_n( B(I,E_n,A); B(I,E_n,A))\\
\xrightarrow{E_n(\epsilon, f)} E_n(A;M)
\xrightarrow{\gamma_M} M
\end{multline*}
for $f\colon B(I,E_n,A) \ra M$, where the first map is defined by using  $B(I,E_n,A) \subset B(E_n,E_n,A) = E_n(B(I,E_n,A))$,
the map $\epsilon \colon B(I,E_n,A) \subset B(E_n,E_n,A) \ra A$ is the standard augmentation,
and $\gamma_M$ is the structure map of the $E_n$-representation $M$ of $A$.
There is a commuting diagram of differential graded modules
\begin{equation*}
\xymatrix{ E_n   ( B(I,E_n,A); B(I,E_n,A)) \ar[d]^-{\cong}\ar[rr]^-{E_n(\epsilon, f)} && E_n(A;M) \ar[rr]^-{\gamma_M} \ar[d]^-{\cong} & &
M \ar@{=}[d]\\
U_{E_n}( B(I,E_n,A)) \otimes  B(I,E_n,A) \ar[rr]^-{U_{E_n}(\epsilon) \otimes f} &&
U_{E_n}(A) \otimes M \ar[r] & U_{\Com}(A) \otimes M \ar[r]^-{\mu_M} &M }
\end{equation*}
with $\mu_M$ defined by $M$ being an $E_n$-representation of $A$. The vertical isomorphisms are given by identifying
\begin{equation*}
p(x_1,\dots,x_{i-1},y,x_{i+1},\dots,x_l) \in \mcP(X;Y)
\end{equation*}
with
\begin{equation*}
p(x_1,\dots,x_{i-1},x,x_{i+1},\dots,x_l) \otimes y \in U_{\mcP}(X) \otimes Y
\end{equation*}
for $p \in \mcP, x_1,\dots,x_l \in X$ and $y \in Y$. Use the identification
\begin{equation*}
(\underline{\Hom}_k(B(I,E_n,A),M), \partial) \cong (\underline{\Hom}_{U_{\Com}(A)}(U_{\Com}(A)\otimes B(I,E_n,A),M), \tilde \partial).
\end{equation*}
The differential  $\tilde \partial(f)$ can be calculated to be induced by $\partial_{\theta_B}$.
Since $(A_+\otimes B(I,E_n,A),\partial_{\theta_B})$ and  $(A_+\otimes \Sigma^{-n}B^n(A),  \partial_{\theta})$ are cofibrant differential graded $A_+$-modules, the quasiisomorphism exhibited in the proof of Theorem \ref{th:EnHomWithUnCoef} induces a quasiisomorphism
from
\begin{equation*}
\underline{\Hom}_{A_+}((A_+\otimes B(I,E_n,A),\partial_{\theta_B}),M)
\end{equation*}
to
\begin{equation*}
\underline{\Hom}_{A_+}((A_+\otimes \Sigma^{-n}B^n(A),  \partial_{\theta}), M).
\end{equation*}
\end{proof}


\begin{thebibliography}{9999999}

\bibitem[BMR]{BMR}
Tobias Barthel, Peter May, Emily Riehl,
\emph{Six model structures for DG-modules over DGAs: Model category theory in homological action}, 2013. Available online at \href{http://arxiv.org/abs/1310.1159}{http://arxiv.org/abs/1310.1159}.

\bibitem[Be96]{Be96}
Clemens Berger,
\emph{Op\'erades cellulaires et espaces de lacets it\'er\'es},
Ann. Inst. Fourier (Grenoble) \textbf{46} (4) 1996, 1125--1157.

\bibitem[Be97]{Be97}
Clemens Berger,
\emph{Combinatorial models for real configuration spaces and {$E_n$}-operads},
Operads: {P}roceedings of {R}enaissance {C}onferences
              ({H}artford, {CT}/{L}uminy, 1995),
Contemp. Math. \textbf{202} 1997, Amer. Math. Soc., Providence, RI, 37--52.

\bibitem[BV73]{BV73}
Michael Boardman, Rainer Vogt,
\emph{Homotopy invariant algebraic structures on topological spaces},
Lecture Notes in Mathematics \textbf{347},
 Springer-Verlag, Berlin (1973), x+257 pp.

\bibitem[EM53]{EM53}
Samuel Eilenberg and Saunders Mac Lane,
\emph{On the groups of {$H(\Pi,n)$}. {I}},
Ann. of Math. (2) \textbf{58} 1953,
55--106.

\bibitem[Fra13]{Fra13}
John Francis,
\emph{The tangent complex and Hochschild cohomology of $\mathcal{E}_n$-rings},
Compos. Math. \textbf{149} 2013, 430--480.

\bibitem[Fre04]{Fre04}
Benoit Fresse, \emph{Koszul duality of operads and homology of partition posets},
 Homotopy theory: relations with algebraic geometry, group
              cohomology, and algebraic {$K$}-theory, Contemp. Math. \textbf{346} 2004, Amer. Math. Soc., Providence, RI, 115--215.

\bibitem[Fre09]{Fre09}
Benoit Fresse,
\emph{Modules over operads and functor},
Lecture Notes in Mathematics \textbf{1967}, Springer-Verlag, Berlin (2009), x+308 pp.

\bibitem[Fre09a]{Fre09a}
Benoit Fresse,
\emph{Operadic cobar constructions, cylinder objects and homotopy morphisms of algebras over operads}, Alpine persepctives on algebraic topology (Arolla, 2008), Contemp. Math. \textbf{504} 2009, Amer. Math. Soc., Providence, RI, 125--189.

\bibitem[Fre11]{Fre11}
Benoit Fresse,
\emph{Iterated bar complexes of E-infinity algebras and homology theories},
Alg. Geom. Topol. \textbf{11} 2011, 747--838.

\bibitem[FrApp]{FrApp}
Benoit Fresse, \emph{Iterated bar complexes and the poset of pruned trees. {A}ddendum to the paper "{I}terated bar complexes of {E}-infinity algebras and homology theories"}, 2008.
Available online at \href{http://math.univ-lille1.fr/~fresse/IteratedBarAppendix.pdf}{http://math.univ-lille1.fr/$\sim$ fresse/IteratedBarAppendix.pdf}.

\bibitem[Hi03]{Hi03}
Philip Hirschhorn,
\emph{Model categories and their localizations},
Mathematical Surveys and Monographs \textbf{99}, American Mathematical Society, Providence, RI (2003), xvi+457 pp.

\bibitem[Ho10]{Ho10}
Eric Hoffbeck, \emph{{$\Gamma$}-homology of algebras over an operad},
Algebr. Geom. Topol. \textbf{10} (3) 2010, 1781--1806.


\bibitem[LR11]{LR11}
Muriel Livernet, Birgit Richter,
\emph{An interpretation of $E_n$-homology as functor homology},
Mathematische Zeitschrift \textbf{269} (1) 2011, 193--219.

\bibitem[Ma72]{Ma72}
Jon Peter May,
\emph{The geometry of iterated loop spaces},
Lectures Notes in Mathematics \textbf{271}, Springer-Verlag, Berlin-New York (1972), viii+175 pp.

\bibitem[Pir00]{Pir00}
Teimuraz Pirashvili,
\emph{Hodge decomposition for higher order Hochschild homology},
Ann. Sci. \'Ecole Norm. Sup. (4) \textbf{33} 2000, 151--179.

\bibitem[RW02]{RW02}
Alan Robinson, Sarah Whitehouse,
\emph{Operads and $\Gamma$-homology of commutative rings},
Math. Proc. Cambridge Philos. Soc. \textbf{132} 2002, 197--234.

\bibitem[Z]{Z}
Stephanie Ziegenhagen, \emph{$E_n$-cohomology with coefficients as functor cohomology}, preprint 2014. Available online at \href{http://arxiv.org/abs/1412.6031}{http://arxiv.org/abs/1412.6031}.

\end{thebibliography}
\end{document}